\documentclass{article}

\usepackage{graphicx}
\usepackage{mathptmx}
\DeclareMathAlphabet{\mathcal}{OMS}{cmsy}{m}{n}

\usepackage[hidelinks]{hyperref}
\usepackage{amsmath,amssymb}
\usepackage{stmaryrd} 
\usepackage{bm}		  
\usepackage{accents}
\usepackage{cancel}
\usepackage[toc,page]{appendix}
\usepackage{subfigure}
\usepackage{booktabs} 
\usepackage{cite}	
\usepackage{epstopdf}

\usepackage{authblk}
\usepackage{amsthm}
\usepackage[margin=3cm]{geometry}
\usepackage{placeins} 

\newcommand{\avg}[1]{\left\{\hspace*{-3pt}\left\{#1\right\}\hspace*{-3pt}\right\}}
\newcommand{\jump}[1]{\ensuremath{\left\llbracket #1 \right\rrbracket}}

\newcommand\iprod[1]{\left\langle #1\right\rangle} 			
\newcommand\iprodN[1]{\left\langle #1\right\rangle_{\!N}}		

\renewcommand\vec[1]{\accentset{\,\rightarrow}{#1}}
\newcommand\spacevec[1]{\accentset{\,\rightarrow}{#1}}		
\newcommand\statevec[1]{\mathbf #1}					
\newcommand\statevecGreek[1]{\boldsymbol #1}			
\newcommand\acclrvec[1]{\accentset{\,\leftrightarrow}{#1}}	
\newcommand\bigstatevec[1]{\acclrvec{{\mathbf #1}}}	
\newcommand\bigcontravec[1]{\acclrvec{\tilde{\mathbf #1}}} 	

\newcommand\normalmatrix[1]{\underline{ #1}}		
\newcommand\dS{\,\operatorname{dS} }

\title{An entropy stable discontinuous Galerkin method for the two-layer shallow water equations on curvilinear meshes}

\author[1]{Patrick Ersing}
\author[1]{Andrew R. Winters}
\affil[1]{Department of Mathematics; Division of Applied Mathematics, Linköping University, Sweden}

\newenvironment{keywords}{\par\textbf{Key words.}}{\par}
\newenvironment{AMS}{\par\textbf{AMS subject classification.}}{\par}

\newtheorem{theorem}{Theorem}[section]
\newtheorem{lemma}[theorem]{Lemma}
\newtheorem{corollary}[theorem]{Corollary}
\theoremstyle{remark}
\newtheorem{remark}[theorem]{Remark}

\begin{document}

\maketitle

\begin{abstract}
We present an entropy stable nodal discontinuous Galerkin spectral element method (DGSEM) for the two-layer shallow water equations on two dimensional curvilinear meshes. We mimic the continuous entropy analysis on the semi-discrete level with the DGSEM constructed on Legendre-Gauss-Lobatto (LGL) nodes. The use of LGL nodes endows the collocated nodal DGSEM with the summation-by-parts property that is key in the discrete analysis. The approximation exploits an equivalent flux differencing formulation for the volume contributions, which generate an entropy conservative split-form of the governing equations. A specific combination of an entropy conservative numerical surface flux and discretization of the nonconservative terms is then applied to obtain a high-order path-conservative scheme that is entropy conservative and has the well-balanced property for discontinuous bathymetry. Dissipation is added at the interfaces to create an entropy stable approximation that satisfies the second law of thermodynamics in the discrete case. We conclude with verification of the theoretical findings through numerical tests and demonstrate results about convergence, entropy stability and well-balancedness of the scheme.

\vspace{0.1cm}


\vspace{0.1cm}

\end{abstract}

\section{Introduction}\label{Introduction}
The two-layer shallow water equations are a widely used model to describe the behavior of two immiscible fluid layers of different density, where the horizontal scales are large relative to the depth scale. Typical applications range from water bodies of different salinity and temperature \cite{castro2007improved} to mixtures of oil and water \cite{fyhn2019consistent}, as well as underwater landslides \cite{ostapenko1999numerical,kim2008two}.

The system is derived from depth-averaging of the inviscid Navier-Stokes equations and can be considered as a combination of the standard shallow water equation (SWE) for each fluid layer with an additional nonconservative term that accounts for the coupling between fluid layers and interaction with the bottom topography.

The design of numerical schemes for this system encounters several difficulties, many of them connected to the inclusion of the nonconservative products. These additional coupling terms affect the wave propagation and can even change the nature of system. In contrast to standard shallow water flows, the system is only conditionally hyperbolic as the baroclinic mode allows for complex eigenvalues. This loss of hyperbolicity occurs for large velocity differences between the layers and has been linked to the onset of Kelvin-Helmholtz instability \cite{abgrall2009two}. The computation of the eigenvalues themselves is also not straightforward as even the one-dimensional system yields a $4 \times 4$ flux Jacobian and a direct calculation of eigenvalues is not possible.

Another important aspect that arises in the numerical treatment of balance laws is the preservation of steady-state solutions. A varying or even discontinuous bottom topography generates additional source contributions that necessitate a careful discretization of source terms and fluxes to ensure a discrete balance. Failure to preserve these steady states can lead to accumulation of balancing errors over time and the formation of unphysical waves \cite{chertock2018well}, on the order of the grid spacing. Since most shallow water applications revolve aorund the computation of small perturbations around a steady-state solution, such as gravity waves, tidal flow or tsunami formation, preserving this balance is essential. A numerical scheme that preserves these steady states is denoted as well-balanced.

In recent developments regarding numerical schemes for the two-layer SWE most activities focused on the development of finite volume methods \cite{castro2007improved,kurganov2009central,swartenbroekx2010hyperbolicity,kim2008two,chiapolino2018models,abgrall2009two} with only few advances in the application of high-order discontinuous Galerkin methods \cite{izem2016discontinuous,fernandez2022arbitrary}.
Due its ability to create high-order approximations and its optimal dissipation and dispersion properties \cite{ainsworth2004dispersive} discontinuous Galerkin methods can provide excellent approximations for this conditionally hyperbolic system of balance laws. In order to utilize these compelling properties, we select DGSEM as our numerical method in this paper. 
While the high-order nature of these methods allows for accurate approximations, the inherent low amount of numerical dissipation also makes them prone to robustness issues \cite{gassner2021novel}.

A contemporary approach to create robust DG approximations in the nonlinear case is the use of entropy stability \cite{tadmor1987numerical}, where the solution is bounded by a strongly convex mathematical entropy function. For the two-layer SWE such an entropy function naturally occurs as the total energy within the layers.
In \cite{gassner2021novel,winters2021construction} a methodology to construct entropy stable DG methods based on the summation-by-parts (SBP) property and a flux-differencing formulation was proposed. Following this approach we construct an entropy stable nodal DGSEM using the SBP properties to mimic a continuous entropy analysis in the semi-discrete case. Furthemore, we show that for a special discretization of fluxes and nonconservative term the resulting scheme is well-balanced.

For the discussion in this paper we first describe the theory and challenges of the two-layer SWE in Section \ref{sec:theory_twolayerSWE}, followed by a continuous entropy analysis of the system in Section \ref{sec:continuous_entropy_analysis}, which we will mimic in the semi-discrete setting. We then derive a split-form DGSEM on curvilinear elements in Section \ref{sec:split_form_dgsem} that will provide a baseline scheme for the further derivation of entropy conservative and entropy stable approximations. In Section \ref{sec:entropy_conservation}, we first derive entropy conservative numerical fluxes that are then used in the split-form DGSEM to create an entropy conservative approximation. We then contract the DGSEM with the entropy variables and rely on the SBP property and metric identities to demonstrate entropy conservation and well-balancedness. From this entropy conservative formulation we then extend to one that is entropy stable by adding appropriate dissipation in Section \ref{sec:entropy_stable_dgsem}. The theoretical findings are then demonstrated and verified in Section \ref{sec:results} with a number of test cases that provide numerical results for convergence, well-balancedness and entropy conservation and stability.

\section{Two-layer equations}\label{sec:theory_twolayerSWE}
We consider the two-layer SWE, which are a nonlinear system of balance laws derived from depth-averaging the inviscid Navier-Stokes-equations. The two-dim\-en\-sion\-al system is given by
\begin{equation}
	\begin{bmatrix}
		h_1 \\ h_1u_1 \\ h_1v_1 \\ h_2 \\ h_2u_2 \\ h_2v_2
	\end{bmatrix}_t
	+
	\begin{bmatrix}
		h_1u_1 \\
		h_1u_1^2 + \frac{gh_1^2}{2}\\
		h_1u_1v_1\\
		h_2u_2 \\
		h_2u_2^2 + \frac{gh_2^2}{2}\\
		h_2u_2v_2
	\end{bmatrix}_x
	+
	\begin{bmatrix}
		h_1v_1 \\
		h_1u_1v_1 \\
		h_1v_1^2 + \frac{gh_1^2}{2}\\
		h_2v_2 \\
		h_2u_2v_2\\
		h_2v_2^2 + \frac{gh_2^2}{2}\\
	\end{bmatrix}_y
	=
	-\begin{bmatrix}
		0 \\ 
		gh_1(b+h_2)_x\\
		gh_1(b+h_2)_y\\
		0 \\
		gh_2(b+\frac{\rho_1}{\rho_2}h_1)_x\\
		gh_2(b+\frac{\rho_1}{\rho_2}h_1)_y\\
	\end{bmatrix},
	\label{eq:two_layer_system_expanded}
\end{equation}
where the subscripts denote the upper $``1"$ and lower ``$2$" fluid layer. For each layer $h_i$ denotes the respective layer height, $u_i$ and $v_i$ the averaged layer velocity in $x$- and $y$-direction and $\rho_i$ the different constant fluid densities, where we assume that $\rho_1<\rho_2$. Furthermore, $b\equiv b(x,y)$ denotes the variable bottom topography and $g$ the gravitational acceleration. The conservative part of the system corresponds to the standard SWE for each layer and only differs by additional nonconservative coupling terms on the right hand side that are given by $gh_1\left(h_2\right)_x$ in the upper layer and $gh_2(\frac{\rho_1}{\rho_2}h_1)_x$ in the lower layer.

To simplify the analysis we introduce a general compact notation for this system of balance laws
\begin{equation}
	\statevec{u}_t + \vec{\nabla}_x \cdot \bigstatevec{f}(\statevec{u}) + \statevecGreek{\phi}(\statevec{u})\circ\left(\vec{\nabla}_x\cdot\bigstatevec{r}(\statevec{u})\right) = 0,
	\label{eq:balance_law_continuous}
\end{equation}
where $\statevec{u} = (h_1, h_1u_1, h_1v_1, h_2, h_2u_2, h_2v_2)^T$ denotes the state vector of conserved quantities and the nonconservative term composed of a nonconservative product and a source term is reformulated in terms of the Hadamard product with the state vector $\statevecGreek{\phi}(\statevec{u}) = (0, gh_1, gh_1, 0, gh_2, gh_2)^T$. To obtain a compact notation for the fluxes and nonconservative term we introduce the block vectors, see \cite{winters2021construction,bohm2020entropy}
\begin{equation}
	\bigstatevec{f} = \begin{bmatrix} \statevec{f}_1 \\ \statevec{f}_2\end{bmatrix}, \quad
	\bigstatevec{r} = \begin{bmatrix} \statevec{r}_1 \\ \statevec{r}_2\end{bmatrix},
\end{equation}
that contain the respective components in $x$- and $y$-directions
\begin{equation}
	\statevec{f}_1 = 
	\begin{bmatrix}
		h_1u_1 \\
		h_1u_1^2 + \frac{gh_1^2}{2}\\
		h_1u_1v_1 \\
		h_2u_2 \\
		h_2u_2^2 + \frac{gh_2^2}{2}\\
		h_2u_2v_2
	\end{bmatrix}\hspace{-0.1cm},\,\,
	\statevec{f}_2 = 
		\begin{bmatrix}
		h_1v_1 \\
		h_1u_1v_1 \\
		h_1v_1^2 + \frac{gh_1^2}{2}\\
		h_2v_2 \\
		h_2u_2v_2\\
		h_2v_2^2 + \frac{gh_2^2}{2}\\
	\end{bmatrix}\hspace{-0.1cm},\,\,
	\statevec{r}_1 =
	\begin{bmatrix}
		0 \\ 
		b+h_2\\
		0\\
		0 \\
		b+\frac{\rho_1}{\rho_2}h_1\\
		0\\
	\end{bmatrix}\hspace{-0.1cm},\,\,
	\statevec{r}_2 =
		\begin{bmatrix}
		0 \\ 
		0\\
		b+h_2\\
		0 \\
		0\\
		b+\frac{\rho_1}{\rho_2}h_1\\
	\end{bmatrix}.
	\label{eq:two_layer_system_fluxes}
\end{equation}

The nonconservative term in \eqref{eq:balance_law_continuous} introduces some difficulties in the design of robust numerical methods. Additionally, in the continuous analysis, the nonconservative term renders the system only conditionally hyperbolic. This escape from the hyperbolic regime has been linked to the onset of Kevin-Helmholtz instabilities \cite{bouchut2010robust}, where high velocity difference in the shear layer leads to heavy mixing. A mechanism that cannot be replicated in the immiscible setting.

Another difficulty arises in the computation of the eigenvalues of the flux Jacobian. Even the one dimensional system yields a fourth order characteristic polynomial, hence a direct calculation of the eigenvalues is not feasible. Instead we follow the approach from Nycander and Döös \cite{nycander2003open} and assume that the baroclinic and barotropic modes do not interact for $u_1\approx u_2$ and $\rho_1\ll\rho_2$. This assumption leads to the eigenvalue approximation 
\begin{equation}
	\begin{aligned}
	\lambda^{\pm}_{ext} &= U_m \pm \sqrt{g(h_1+h_2)},\\
	\lambda^{\pm}_{int} &= U_c \pm \sqrt{g'\frac{h_1h_2}{h_1+h_2} \left(1 - \frac{\left(u_1-u_2\right)^2}{g'\left(h_1+h_2\right)}\right)},
	\end{aligned}
	\label{eq:eigenvalue_approximation}
\end{equation}
where $\lambda_{ext}$ and $\lambda_{int}$ denote the wave speeds of the barotropic and baroclinic mode, respectively written in terms of the velocities $U_m = \frac{h_1u_1+h_2u_2}{h_1+h_2}$, $U_c = \frac{h_1u_2+h_2u_1}{h_1+h_2}$ and the reduced gravity $\smash{g' = g(1 - \frac{\rho_1}{\rho_2})}$. For simplicity the one-dimensional case is considered as the extension to two dimensions is straightforward. From the approximation \eqref{eq:eigenvalue_approximation} we see that internal eigenvalues may become complex and the hyperbolic regime of the system is estimated by
\begin{equation}
	\frac{\left(u_1-u_2\right)^2}{g'\left(h_1+h_2\right)} \leq 1.
\end{equation}

In accordance with \cite{fjordholm2012energy} we can get the following bound on the maximum wave speeds
\begin{equation}
	|\lambda| \leq \left|\frac{h_1u_1+h_2u_2}{h_1+h_2}\right| \pm \sqrt{g(h_1+h_2)}.
	\label{eq:bound_eigenvalue}
\end{equation}

A particularly important property of numerical schemes for the two-layer SWE is the preservation of steady-state solutions that remain constant in time. The most important example of such a steady-state is the lake-at-rest condition 
\begin{equation}
	u_1,u_2,v_1,v_2 \equiv 0, \quad h_2+b \equiv \textrm{constant}, \quad h_1 \equiv \textrm{constant}.
	\label{eq:lake_at_rest_condition}
\end{equation}
As a wide range of applications are computed from small perturbations around the lake-at-rest condition, discrete preservation of \eqref{eq:lake_at_rest_condition} is essential. A scheme that is able to preserve \eqref{eq:lake_at_rest_condition} is denoted as well-balanced.

\section{Entropy analysis}\label{sec:continuous_entropy_analysis}
Stability estimates for numerical approximations are typically provided in terms of $L^2$-stability, where the discrete solution remains bounded within the $L^2\text{-norm}$, see \cite{kreiss1970,oliger1978}. While $L^2$-stability holds for linear problems, it fails to provide stability for general nonlinear systems \cite{gassner2021novel,merriam1989entropy}. An alternative for nonlinear systems is the use of entropy stability \cite{tadmor1987numerical}, where the solution is bounded by a strongly convex mathematical entropy function. The two-layer shallow water system \eqref{eq:two_layer_system_expanded} is equipped with such a convex entropy function $S=S(\statevec{u})$, corresponding to the total energy within the layers
\begin{equation}
	\begin{aligned}
	S(\statevec{u}) =& \dfrac{1}{2}\left(\rho_1 \left(h_1u_1^2+h_1v_1^2+gh_1^2\right) + \rho_2 \left(h_2u_2^2+h_2v_2^2+gh_2^2\right)\right)\\
	&+ \rho_2gh_2b + \rho_1gh_1(b+h_2),
	\end{aligned}
	\label{eq:entropy_function}
\end{equation}
together with the corresponding entropy flux
\begin{equation}
	\begin{aligned}
	\spacevec{f}^{\,S} = 
	&\rho_1\left(\frac{1}{2}h_1u_1\left(u_1^2 + v_1^2\right) + gh_1u_1\left(h_1 + b\right)\right)\\
   +&\rho_2\left(\frac{1}{2}h_2u_2\left(u_2^2 + v_2^2\right) + gh_2u_2\left(h_2 + b\right)\right)
   +g\rho_1h_1h_2(u_1 + u_2).
   \end{aligned}
\end{equation}
Differentiation of the entropy function with respect to the conservative variables yields the entropy variables 
\begin{equation}
	\statevec{w}^T = S_u =  \left[\begin{matrix}\rho_{1}\left( g \left(h_{1} + h_{2} + b \right) - \frac{1}{2} \left( u_{1}^{2} + v_{1}^{2}\right)\right)\\
	\rho_{1} u_{1}\\
	\rho_{1} v_{1}\\
	\rho_{2}\left(g\left(h_{1} \frac{\rho_{1}}{\rho_2} + h_{2} + b \right) - \frac{1}{2} \left(u_{2}^{2} + v_{2}^{2}\right)\right)\\\rho_{2} u_{2}\\ \rho_{2} v_{2}\end{matrix}\right],
	\label{eq:entropy_variables}
\end{equation}
which are used to obtain a relation between the physical and entropy fluxes as well as the nonconservative terms
\begin{equation}
	\spacevec{\nabla}_x\cdot\spacevec{f}^{\,S} = \statevec{w}^T\left(\spacevec{\nabla}_x\cdot\bigstatevec{f} + \statevecGreek{\phi}\circ\left(\vec{\nabla}_x\cdot\bigstatevec{r}\,\right)\right).
	\label{eq:entropy_flux_def}
\end{equation}
Furthermore, the entropy flux potential is defined as
\begin{equation}
	\spacevec{\psi}\left(\statevec{u}\right) = \statevec{w}^T\bigstatevec{f} - \spacevec{f}^{\,S}.
	\label{eq:entropy_flux_potential}
\end{equation}
From these compatibility relations, contracting smooth solutions of \eqref{eq:balance_law_continuous} with the entropy variables $\statevec{w}$ yields a scalar conservation law for the entropy function
\begin{equation}
	S_t + \spacevec{\nabla}_x\cdot\spacevec{f}^{\,S} = 0.
	\label{eq:entropy_conservation_law}
\end{equation}
We integrate the equation over the domain $\Omega$ and apply Gauss' law to rewrite the entropy flux as a surface integral and obtain the entropy conservation law \eqref{eq:entropy_conservation_law} in integral form
\begin{equation}
	\int\limits_{\Omega}S_t \operatorname{dV} + \int\limits_{\partial \Omega}\spacevec{f}^{\,S}\cdot\spacevec{n} \dS = 0.
	\label{eq:entropy_conservation_law_integral_form}
\end{equation}
In the case of discontinuities entropy must be dissipated at shocks and solutions must instead satisfy the entropy inequality
\begin{equation}
	\int\limits_{\Omega}S_t \operatorname{dV} + \int\limits_{\partial \Omega}\spacevec{f}^{\,S} \cdot \spacevec{n} \dS \leq 0.
	\label{eq:entropy_inequality}
\end{equation} 

To prove entropy conservation of the numerical scheme, we will mimic the continuous case. Note, that the contraction of physical fluxes into entropy space depends on the chain rule, which in general does not hold discretely \cite{ranocha2019mimetic}. In order to obtain entropy fluxes in the discrete setting, we therefore employ an alternative strategy based on entropy conservative numerical finite volume fluxes, that will be described in Section \ref{sec:entropy_conservation}.

\section{Split-Form DGSEM on curvinlinear elements}\label{sec:split_form_dgsem}
In this section we derive the split-form DGSEM for the nonconservative system \eqref{eq:balance_law_continuous}, that will provide a baseline scheme from which we then develop our entropy conservative and stable approximations.
To generalize the formulation to curvilinear elements we first introduce a mapping between physical and reference space. We then formulate the DGSEM for the nonconservative system and replace the volume integrals with a flux-differencing formulation to obtain the split-form DGSEM \cite{gassner2016split}.
\subsection{Mapping the equations}
We first subdivide the physical domain $\Omega$ into $K$ non-overlapping unstructured quadrilateral elements. Each element is mapped from physical space onto a unit square element $E=[-1,1]^2$ in computational space using transfinite interpolation with linear blending \cite{kopriva2009implementing}. The mapping is given as
\begin{equation}
	\vec{x} = \mathcal{X}(\vec{\xi}) \equiv X\hat{x} + Y\hat{y}
\end{equation}
 with the coordinates in physical space $\vec{x} = (x, y)^T = (x_1, x_2)$ and computational space $\vec{\xi} = (\xi, \eta)^T = (\xi^1, \xi^2)$, respectively. From the mapping we directly calculate the covariant basis vectors
 \begin{equation}
	\begin{aligned}
		\vec{a}_1 &= \frac{\partial \mathcal{X}}{\partial \xi} = X_{\xi}\hat{x} + Y_{\eta}\hat{y}\\
		\vec{a}_2 &= \frac{\partial \mathcal{X}}{\partial \eta} = X_{\eta}\hat{x} + Y_{\xi}\hat{x}\\
		\vec{a}_3 &= \hat{z}
	\end{aligned}
\end{equation}
that are tangential to the coordinate lines in computational space. The covariant basis vectors are then used to compute the Jacobian of the mapping
\begin{equation}
	J = \vec{a}_1 \cdot (\vec{a}_2 \times \vec{a}_3) = X_{\xi}Y_{\eta} - Y_{\xi}X_{\eta}
\end{equation}
as well as the contravariant basis vectors normal to the coordinate lines
\begin{equation}
	\begin{aligned}
		J\vec{a}_1 &= \vec{a}_2 \times \vec{a}_3 = Y_{\eta}\hat{x} - X_{\eta}\hat{y}, \\
		J\vec{a}_2 &= \vec{a}_3 \times \vec{a}_1 = -Y_{\xi}\hat{x} + X_{\xi}\hat{y}.
		\label{eq:metric_terms}
	\end{aligned}
\end{equation}

To map the system of equations into computational space, we first define a block matrix with the metric terms \eqref{eq:metric_terms} together with a $6\times6$ identity matrix $\normalmatrix{I}$ to simplify the notation
\begin{equation}
	\mathfrak{M} = \begin{bmatrix}
		Ja_1^1\normalmatrix{I} & Ja_1^2\normalmatrix{I} \\
		Ja_2^1\normalmatrix{I} & Ja_2^2\normalmatrix{I}
	\end{bmatrix}.
\end{equation}
We transform the  divergence operator into computational space
\begin{equation}
	\spacevec{\nabla}_x \cdot \bigstatevec{f} = 
	 \frac{1}{J} \spacevec{\nabla}_{\xi} \cdot \bigcontravec{f},
	\quad 
	\statevecGreek{\phi} \circ \left(\spacevec{\nabla}_{x} \cdot \bigstatevec{r}\,\right) =
	\statevecGreek{\phi} \circ \left(\spacevec{\nabla}_{\xi} \cdot \bigcontravec{r}\,\right)
	\label{eq:transform_flux_div}
\end{equation}
using the contravariant block vectors for the flux and nonconservative terms
\begin{equation}
	\bigcontravec{f} = \mathfrak{M}^T \bigstatevec{f}, \quad \bigcontravec{r} = \mathfrak{M}^T \bigstatevec{r}.
\end{equation}
Applying transformation \eqref{eq:transform_flux_div} to \eqref{eq:balance_law_continuous} then yields the transformed balance law in computational space
\begin{equation}
	J\statevec{u}_t + \vec{\nabla}_{\xi} \cdot \bigcontravec{f} 
	+ \statevecGreek{\phi} \circ \left(\spacevec{\nabla}_{\xi} \cdot \bigcontravec{r}\,\right) = 0.
	\label{eq:balance_law_reference_space}
\end{equation}

In order to obtain a free-stream preserving numerical method, the divergence of a constant flux must vanish such that the metric identities 
\begin{equation}
	\frac{\partial}{\partial \xi}J\spacevec{a}^{\,1} + \frac{\partial}{\partial \eta} J\spacevec{a}^{\,2} = 0
	\label{eq:metric_identity}
\end{equation}
are satisfied discretely. That the metric identities hold in the discrete setting is a prerequisite for the preservation of steady-state solutions and the entropy conservation proof presented later in this work. In two space dimensions this is guaranteed for isoparametric boundaries, where the same polynomial order (or lower) from the DGSEM approximation is use to construct the mapping \cite{kopriva2006metric}.

\subsection{Nodal discontinuous Galerkin Spectral Element method}
The numerical scheme developed in this work is based on a standard nodal collocation DGSEM formulation as described in the literature, see \cite{kopriva2009implementing,hesthaven2007nodal}. In the following we formulate our DGSEM and extend it for nonconservative terms analogous to the numerical schemes developed in \cite{renac2019entropy,franquet2012runge}.

We begin with the derivation of a weak formulation of \eqref{eq:balance_law_reference_space}, where we multiply the transformed equation with test functions $\statevec{\varphi}$ and integrate over the domain $\Omega$. For nonconservative systems the definition of weak solutions is rather difficult as the theory of distributions does not apply and classical conservative approaches may recover incorrect shock speeds. Instead we follow the derivation provided by Franquet and Perrier \cite{franquet2012runge}, where a weak formulation is obtained from integration-by-parts for the conservative fluxes, while the nonconservative term is defined as a Borel measure according to theory of DalMaso, LeFloch and Murat \cite{dal1995definition}. For smooth solutions this measure corresponds to classical integration, while at discontinuities it must be evaluated as a line integral depending on a family of paths $\upsilon$. The resulting weak formulation is then given by
\begin{equation}
	\iprod{J\statevec{u}_t,\statevecGreek{\varphi}} 
	+ 
	\int\limits_{\partial E} \statevecGreek{\varphi}^T \left(\bigstatevec{f} \cdot \vec{n}\right) \hat{s} \,\text{d}S 
	-
	\iprod{\bigcontravec{f},\spacevec{\nabla}_{\xi}\statevecGreek{\varphi}} 
	+\int\limits_{\partial E} \statevecGreek{\varphi}^T \left(\statevecGreek{\phi}\circ\bigstatevec{r}\,\right)^{\diamond}\cdot\spacevec{n} \hat{s} \,\text{d}S
	+
	\iprod{\statevecGreek{\phi}\circ\left(\spacevec{\nabla}_{\xi}\cdot\bigcontravec{r}\,\right),\statevecGreek{\varphi}}
	= 0,
	\label{eq:weak_form_continuous}
\end{equation}
where $\hat{s}$ denotes the differential surface element given by
\begin{equation}
	\begin{aligned}
		\xi = \pm 1 : \hat{s}(\eta) = \left|J\spacevec{a}^{\,1}(\pm 1,\eta)\right|, \quad \eta = \pm 1 : \hat{s}(\xi) = \left|J\spacevec{a}^{\,2}(\xi,\pm 1)\right|,
	\end{aligned}
\end{equation}
and $\vec{n}$ is the outwards pointing unit normal vector in the physical space
\begin{equation}
	\spacevec{n} = \frac{J\spacevec{a}^{\,i}}{\hat{s}}, \quad i=1,2.
\end{equation} 
We introduce a notation where at each interface we denote the state of the primary element as $``-"$ and the state of the secondary (neighboring) element as $``+"$. Accordingly, the physical normal vector $\spacevec{n}$ is defined to point from $``-"$ to $``+"$, which the following relation for the normal fluxes of neighboring elements at an interface
\begin{equation}
	\vec{n} = \vec{n}^{\,-} = -\vec{n}^{\,+}.
	\label{eq:defintion_normal_vector_pm}
\end{equation}
Furthermore, we introduced the inner product notation on the reference element, given for state and block vectors
\begin{equation}
	\iprod{\statevec{u},\statevec{v}} = \int\limits_{E} \statevec{u}^T \statevec{v} \operatorname{d\xi}\operatorname{d\eta}, \qquad
	\iprod{\bigstatevec{f},\bigstatevec{g}} = \int\limits_{E}\sum_{i=1}^2 \statevec{f}_i^{\,T} \statevec{g}_i \operatorname{d\xi}\operatorname{d\eta}
	\label{eq:inner_product}
\end{equation}
respectively.

To account for the definition of the nonconservative term, we further introduced the surface numerical nonconservative term $\statevecGreek{\varphi}^T \left(\statevecGreek{\phi}\circ\bigstatevec{r}\,\right)^{\diamond}$ that satisfies the consistency condition
\begin{equation}
 \left(\statevecGreek{\phi}\circ\bigstatevec{r}\,\right)^{\diamond}(\statevec{u},\statevec{u}) = 0
\end{equation}
and the path-conservative property \cite{pares2006numerical} so that at each element interface we obtain
\begin{equation}
 \left(\statevecGreek{\phi}\circ\bigstatevec{r}\,\right)^{\diamond}(\statevec{u}^-,\statevec{u}^+)\cdot\spacevec{n}^{\,-}
 +
 \left(\statevecGreek{\phi}\circ\bigstatevec{r}\,\right)^{\diamond}(\statevec{u}^+,\statevec{u}^-)\cdot\spacevec{n}^{\,+}=
	\int\limits_0^1 {A}^{NC}(\upsilon(s;\statevec{u}^-,\statevec{u}^+))\partial_s \upsilon(s;\statevec{u}^-,\statevec{u}^+)\operatorname{ds},
\label{eq:path_conservative_property}
\end{equation}
where ${A}^{NC}$ denotes the flux Jacobian of the nonconservative subsystem and $\upsilon$ is a path connecting neighboring states $\statevec{u}^-$ and $\statevec{u}^+$ across a discontinuity, for complete details see \cite{pares2006numerical}. The numerical nonconservative term in \eqref{eq:weak_form_continuous} is therefore equivalent to a fluctuation form \cite{pares2006numerical}, where for each element $\left(\statevecGreek{\phi}\circ\bigstatevec{r}\,\right)^{\diamond}\cdot\spacevec{n}$ recovers the contributing fluctuations at the interface.

Next, we approximate the solution with a local tensor-product basis of polynomial degree $N$
\begin{equation}
	u \approx U = \sum\limits_{i,j=0}^N U_{ij} \ell_i(\xi) \ell_j(\eta),
\end{equation}
where the basis is spanned by one-dimensional nodal Lagrange basis functions
\begin{equation}
	\ell_j(\xi) = \prod\limits_{i=0,i \neq j}^N \frac{\xi - \xi_i}{\xi_j - \xi_i}, \quad i,j=0,...,N,
	\label{eq:Lagrange_basis}
\end{equation} 
on the interval $\xi \in [-1,1]$ with $N+1$ interpolation nodes $\{\xi_i\}_{i=0}^N$, located at the Legendre-Gauss-Lobatto (LGL) points. 
In the following, interpolated quantities at the LGL nodes are denoted with capital letters and the interpolation of a function by $G = \mathbb{I}^N(g)$. We then use the basis functions to define the discrete derivative operator
\begin{equation}
	\mathcal{D}_{ij}:=\left.\frac{\partial\ell_j}{\partial\xi}\right|_{\xi=\xi_i}, \quad i,j=0,...,N,
	\label{eq:Differentiation_matrix}
\end{equation}
apply a LGL quadrature rule to approximate the integrals present in the weak formulation \eqref{eq:weak_form_continuous} and collocate the interpolation and quadrature nodes. This yields the diagonal mass matrix
\begin{equation}  
	\mathcal{M} = \text{diag}(\omega_{\hspace{0.5pt}0},...,\omega_{N}),
	\label{eq:mass_matrix}
\end{equation}
with the LGL quadrature weights $\{\omega_i\}_{i=0}^N$. Using the Kronecker delta property of the Lagrange basis functions \eqref{eq:Lagrange_basis} we then introduce the discrete inner product notation
\begin{equation}
	\iprod{f,g}_N = \sum_{i,j=0}^N f_{ij} g_{ij} \omega_{i}\omega_{j} \equiv \sum_{i,j=0}^N f_{ij} g_{ij} \omega_{ij}.
	\label{eq:discrete_inner_product}
\end{equation}
This specific choice of interpolation and differentiation operators possesses the diagonal norm SBP property for all polynomial orders $N$ \cite{gassner2013skew}
\begin{equation}
	(\mathcal{M}\mathcal{D}) + (\mathcal{M}\mathcal{D})^T = \mathcal{Q} + \mathcal{Q}^T = \mathcal{B},
	\label{eq:SBP_integration_by_parts}
\end{equation}
where we introduced the SBP matrix $\mathcal{Q}$ and the boundary matrix $\mathcal{B}$
\begin{equation}
	\mathcal{Q}=\mathcal{M}\mathcal{D}, \qquad \mathcal{B} = \text{diag}(-1,0,...,0,1).
	\label{eq:definition_sbp_matrices}
\end{equation}
This property will be necessary to prove entropy stability, as it allows us to mimic integration-by-parts and enables the extended Gauss Law in the discrete case \cite{kopriva2017polynomial}.

We introduce polynomial approximations for the solution variables, fluxes, test functions and the nonconservative terms as well as LGL quadrature to obtain a discrete version of \eqref{eq:weak_form_continuous}
\begin{equation}
	\begin{aligned}
	\iprodN{\mathbb{I}^N(J)\statevec{U}_t,\statevecGreek{\varphi}} 
	+ 
	&\int\limits_{\partial E,N} \statevecGreek{\varphi}^T \statevec{F}_n \hat{s} \,\dS - \iprodN{\bigcontravec{F} ,\spacevec{\nabla}_{\xi}\statevecGreek{\varphi}} 
	\\+
	&\int\limits_{\partial E,N} \statevecGreek{\varphi}^T \left(\statevecGreek{\Phi}\circ\bigstatevec{R}\right)^{\diamond}\cdot\spacevec{n} \hat{s} \,\text{d}S
	+
	\iprodN{\statevecGreek{\Phi}\circ\left(\spacevec{\nabla}_{\xi}\cdot\mathbb{I}^N(\bigcontravec{R})\right),\statevecGreek{\varphi}}
	= 0,
	\end{aligned}
\end{equation}
where capital letters denote the discrete approximation and the surface fluxes $\smash{\bigstatevec{F}\cdot \vec{n}}$ and numerical nonconservative terms ${(\statevecGreek{\Phi}\circ\bigstatevec{R})^{\diamond}\cdot\spacevec{n}}$ are replaced by the surface normal fluxes $\smash{\statevec{F}_n}$ and surface numerical nonconservative terms $\smash{(\statevecGreek{\Phi}\circ\bigstatevec{R})^{\diamond}_n}$. Since the physical fluxes are not uniquely defined at the discontinuous interfaces, we replace the surface normal flux $\statevec{F}_n$ by a surface numerical flux $\statevec{F}_n^*$, to resolve the discontinuity such that we obtain the discrete weak formulation
\begin{equation}
	\begin{aligned}
		\iprodN{\mathbb{I}^N(J)\statevec{U}_t,\statevecGreek{\varphi}} + 
		&\int\limits_{\partial E,N} \statevecGreek{\varphi}^T \statevec{F}^*_n \hat{s} \,\dS 
		- 
		\iprodN{\bigcontravec{F} ,\spacevec{\nabla}_{\xi}\statevecGreek{\varphi}} 
		\\+
		&\int\limits_{\partial E,N} \statevecGreek{\varphi}^T \left(\statevecGreek{\Phi}\circ\bigstatevec{R}\right)^{\diamond}_n \hat{s} \,\text{d}S
		+
		\iprodN{\statevecGreek{\Phi}\circ\left(\spacevec{\nabla}_{\xi}\cdot\mathbb{I}^N(\bigcontravec{R})\right),\statevecGreek{\varphi}}
		= 0.
	\end{aligned}
\end{equation}

We then mimic integration-by-parts with the the SBP property \eqref{eq:SBP_integration_by_parts} and apply the discrete extended Gauss Law from \cite{kopriva2017polynomial} on the flux terms to obtain the strong form DGSEM approximation of \eqref{eq:balance_law_reference_space}
\begin{equation}
	\begin{aligned}
		\iprodN{\mathbb{I}^N(J)\statevec{U}_t,\statevecGreek{\varphi}} + 
		&\int\limits_{\partial E,N} \statevecGreek{\varphi}^T \{ \statevec{F}_n^* - \statevec{F}_n\} \hat{s} \dS 
		+ \iprodN{\vec{\nabla}_{\xi}\cdot\mathbb{I}^N(\bigcontravec{F}) ,\statevecGreek{\varphi}} 
		\\+
		&\int\limits_{\partial E,N} \statevecGreek{\varphi}^T \left(\statevecGreek{\Phi}\circ\bigstatevec{R}\right)^{\diamond}_n \hat{s} \,\text{d}S
		+
		\iprodN{\statevecGreek{\Phi}\circ\left(\spacevec{\nabla}_{\xi}\cdot\mathbb{I}^N(\bigcontravec{R})\right),\statevecGreek{\varphi}}
		= 0.
	\end{aligned}
	\label{eq:strong_form_discrete}
\end{equation}

The method \eqref{eq:strong_form_discrete} serves as a baseline from which we will continue to construct an entropy stable approximation using a special combination of differencing operator, numerical flux, and numerical nonconservative term.

\subsection{Flux-differencing formulation}
To obtain an entropy stable discretization, we make use of another property that diagonal norm SBP operators have and rewrite the flux divergence term in the volume into a flux differencing formulation \cite{gassner2016split,carpenter2014entropy}
\begin{equation}
	\begin{aligned}
		\vec{\nabla}_{\xi}\cdot\mathbb{I}^N(\bigcontravec{F}) \approx
		\vec{\mathbb{D}} \cdot  \bigcontravec{F}^{\#} 
		=
		&2\sum_{m=0}^{N} \mathcal{D}_{im} \left(\bigstatevec{F}^{\#}(\statevec{U}_{ij}, \statevec{U}_{mj}) \cdot \avg{J\vec{a}^{\,1}}_{(i,m)j} \right)
		\\ +
		&2\sum_{m=0}^{N} \mathcal{D}_{jm}   \left(\bigstatevec{F}^{\#}(\statevec{U}_{ij}, \statevec{U}_{im}) \cdot \avg{J\vec{a}^{\,2}}_{i(j,m)} \right),
		\label{eq:split_form_flux}
	\end{aligned}
\end{equation}
with the consistent and symmetric two-point flux $\statevec{F}^{\#}$ and arithmetic mean values given by
\begin{equation}
	\avg{\cdot}_{(i,m)j} = \frac{1}{2} \left(\left(\cdot\right)_{ij} + (\cdot)_{mj}\right), \quad \avg{\cdot}_{i\left(j,m\right)} = \frac{1}{2} \left(\left(\cdot\right)_{ij} + \left(\cdot\right)_{im}\right).
\end{equation}

This formulation was originally found by Fisher et al. \cite{fisher2013discretely}, who showed that a diagonal norm SBP operator can be rewritten into a finite volume type flux-differencing formulation on a staggered grid. From a particular choice of the two-point volume flux, the formulation \eqref{eq:split_form_flux} recovers different alternative split-forms of the original partial differential equation that can improve robustness of the scheme \cite{gassner2016split}. The formulation in \eqref{eq:split_form_flux} also extends two-point finite-volume fluxes to high-order \cite{gassner2016split} and recovers an entropy conservative volume discretization 
if an entropy conservative two-point flux is used \cite{carpenter2014entropy,gassner2021novel}. Recently, this approach has been applied successfully to create entropy stable DG methods for the shallow water equation \cite{wintermeyer2017entropy,wu2021high,gassner2016well}, magnetohydrodynamics equations \cite{bohm2020entropy,rueda2023entropy} and the compressible Navier-Stokes equations \cite{winters2021construction,chan2022entropy}.

In \cite{renac2019entropy,waruszewski2022entropy} it was shown that the flux-differencing formulation can also be extended to nonconservative systems such that we can apply the same discretization technique on the conservative and nonconservative parts of the equation. This ensures that both terms are evaluated in the same way, which will be necessary to obtain entropy conservation and discrete preservation of steady-state solutions. Therefore, we introduce the following flux-differencing formulation for the nonconservative volume term
\begin{equation}
		\begin{aligned}
			\statevecGreek{\Phi}\circ\left(\spacevec{\nabla}_{\xi}\cdot\mathbb{I}^N(\bigcontravec{R})\right)
			\approx
			\spacevec{\mathbb{D}}\cdot(\statevecGreek{\Phi}\circ\bigcontravec{R})^{\#} = 
			&\sum_{m=0}^{N} 2\mathcal{D}_{im}\left( (\statevecGreek{\Phi}\circ\bigstatevec{R})^{\#}(\statevec{U}_{ij},\statevec{U}_{mj}) \cdot \avg{J\vec{a}^{\,1}}_{(i,m)j} \right)
			\\ +
			&\sum_{m=0}^{N} 2\mathcal{D}_{jm}\left( (\statevecGreek{\Phi}\circ\bigstatevec{R})^{\#}(\statevec{U}_{ij},\statevec{U}_{im}) \cdot \avg{J\vec{a}^{\,2}}_{i(j,m)} \right),
		\end{aligned}
	\label{eq:split_form_noncons}
\end{equation}
where $(\statevecGreek{\Phi}\circ\bigstatevec{R})^{\#}$ denotes a two-point volume numerical nonconservative term with a given structure
\begin{equation}
	\begin{aligned}
	(\statevecGreek{\Phi}\circ\bigstatevec{R})^{\#}(\statevec{U}^{-},\statevec{U}^{+}) = &\statevecGreek{\Phi}^{\#}(\statevec{U}^{-},\statevec{U}^{+})\circ\bigstatevec{R}^{\#}(\statevec{U}^{-},\statevec{U}^{+}) 
	+\statevecGreek{\Phi}^{\#}(\statevec{U}^{+},\statevec{U}^{-})\circ\bigstatevec{R}^{\#}(\statevec{U}^{+},\statevec{U}^{-}) 
	\\&- \left(
	\statevecGreek{\Phi}^{\#}(\statevec{U}^{-},\statevec{U}^{+}) +
	\statevecGreek{\Phi}^{\#}(\statevec{U}^{+},\statevec{U}^{-})
	\right) \circ \bigstatevec{R}(\statevec{U}^{-}),
	\end{aligned}
	\label{eq:numerical_noncons_structure}
\end{equation}
where $\statevecGreek{\Phi}^{\#}$ and $\bigstatevec{R}^{\#}$ are vector-valued functions.
It was shown in \cite[Lemma 2]{waruszewski2022entropy} that under the consistency conditions
\begin{equation}
	\begin{aligned}
	2\statevecGreek{\Phi}^{\#}(\statevec{U},\statevec{U}) &= \statevecGreek{\Phi}(\statevec{U}) \\
	\bigstatevec{R}^{\#}(\statevec{U},\statevec{U}) &= \bigstatevec{R}(\statevec{U})
	\end{aligned}
	\label{eq:numerical_noncons_consistency_cond2}
\end{equation}
the volume numerical nonconservative term \eqref{eq:numerical_noncons_structure} provides a consistent volume discretization of the nonconservative term. For more details about the flux-differencing formulation for nonconservative terms we refer to \cite{renac2019entropy,waruszewski2022entropy}.

Inserting the flux-differencing formulations \eqref{eq:split_form_flux} and \eqref{eq:split_form_noncons} into the discrete strong formulation \eqref{eq:strong_form_discrete} then yields the split-form DGSEM
\begin{equation}
	\begin{aligned}
	\iprodN{\mathbb{I}^N(J)\statevec{U}_t,\statevecGreek{\varphi}} + 
	&\int\limits_{\partial E,N} \statevecGreek{\varphi}^T \{ \statevec{F}_n^* - \statevec{F}_n\} \hat{s} \dS + \iprodN{\vec{\mathbb{D}}\cdot\bigcontravec{F}^{\#} ,\statevecGreek{\varphi}} 
	\\+
	&\int\limits_{\partial E,N} \statevecGreek{\varphi}^T \left(\statevecGreek{\Phi}\circ\bigstatevec{R}\right)^{\diamond}_n \hat{s} \,\text{d}S
	+
	\iprodN{\spacevec{\mathbb{D}}\cdot(\statevecGreek{\Phi}\circ\bigcontravec{R})^{\#},\statevecGreek{\varphi}} = 0.
	\end{aligned}
	\label{eq:split_formulation}
\end{equation}

With method \eqref{eq:split_formulation} we are now equipped to derive the necessary components $\statevec{F}^{\,\#}$, $\statevec{F}^{\,*}$ and $(\statevecGreek{\Phi} \circ \bigstatevec{R})^{\#}$ , $(\statevecGreek{\Phi} \circ \bigstatevec{R})^{\diamond}$ to first create an entropy conservative and well-balanced scheme, and from that, a scheme that is entropy stable.

\section{Entropy-Conservative DGSEM}\label{sec:entropy_conservation}
To show that the scheme is entropy conservative we use the approach described in \cite{bohm2020entropy,winters2021construction,wintermeyer2017entropy} and mimic the continuous entropy analysis of Section \ref{sec:continuous_entropy_analysis} in the discrete setting. First, we contract the split-form DGSEM \eqref{eq:split_formulation} into entropy space by setting the test function to the interpolated entropy variables $\statevec{W}$
\begin{equation}
	\begin{aligned}
	\iprodN{\mathbb{I}^N(J)\statevec{U}_t,\statevec{W}} + 
	&\int\limits_{\partial E,N} \statevec{W}^T \{ \statevec{F}_n^* - \statevec{F}_n\} \hat{s} \dS + \iprodN{\vec{\mathbb{D}}\cdot\bigcontravec{F}^{\#} ,\statevec{W}}  \\+
	&\int\limits_{\partial E,N} \statevec{W}^T \left(\statevecGreek{\Phi}\circ\bigstatevec{R}\right)^{\diamond}_n \hat{s} \,\text{d}S
	+
	\iprodN{\spacevec{\mathbb{D}}\cdot(\statevecGreek{\Phi}\circ\bigcontravec{R})^{\#},\statevec{W}} = 0.
	\end{aligned}
	\label{eq:split_formulation_entropy_space}	
\end{equation}

\subsection{Time derivative}
As we are interested in the semi-discrete analysis, we assume that the chain rule with respect to differentiation in time holds and examine the time derivative term in \eqref{eq:split_formulation_entropy_space} to obtain
\begin{equation}
	\iprodN{\mathbb{I}^N(J)\statevec{U}_t,\statevec{W}}
	=
	\sum\limits_{i,j=0}^{N} J_{ij}\omega_{ij}\statevec{W}_{ij}^T\frac{\operatorname{d}\statevec{U}_{ij}}{\operatorname{dt}}=
	\sum\limits_{i,j=0}^{N} J_{ij}\omega_{ij}\frac{\operatorname{d}S_{ij}}{\operatorname{dt}}
	= 
	\iprodN{\mathbb{I}^N(J)S_t,1}.
	\label{eq:discrete_entropy_analysis_time_derivative}
\end{equation}
From \eqref{eq:discrete_entropy_analysis_time_derivative} we see that the rate of entropy change for a single element depends solely on the remaining surface and volume contributions. To obtain the total entropy, we sum over the all  $K$ elements in the domain $\Omega$
\begin{equation}
	\frac{\operatorname{d\bar{S}}}{\operatorname{dt}} = \sum_{k=1}^K \iprodN{\mathbb{I}^N(J)^kS_t^k,1}.
	\label{eq:total_discrete_entropy}
\end{equation}

To demonstrate discrete entropy conservation, we will show that the entropy conservation law \eqref{eq:entropy_conservation_law} is satisfied in a discrete sense by the split-form approximation \eqref{eq:split_formulation}. Assuming a closed system, the second law of thermodynamics states that the entropy flux at the boundaries vanishes and the total entropy must be conserved
\begin{equation}
	\frac{\operatorname{d\bar{S}}}{\operatorname{dt}} = 0.
	\label{eq:total_discrete_entropy_conservation_statement}
\end{equation}

\subsection{Discrete entropy flux}
The contraction of the physical flux nonconservative terms to the entropy flux in \eqref{eq:entropy_flux_def} is challenging as it depends on the chain rule, which in general does not hold in the discrete setting \cite{ranocha2019mimetic}. Therefore, we apply an alternative strategy to find a combination of numerical fluxes and numerical nonconservative terms that recover the entropy flux on a discrete level. To reduce the degrees of freedom in this derivation we assume a fixed discretization of the numerical nonconservative term
\begin{equation}
	(\statevecGreek{\Phi}\circ\bigstatevec{R})^{\diamond} 
	= \frac{1}{2}\statevecGreek{\Phi}^-\circ\jump{\bigstatevec{R}} 
	= \frac{1}{2}\statevecGreek{\Phi}^-\circ\left(\bigstatevec{R}^+ - \bigstatevec{R}^-\right),
	\label{eq:definition_numerical_noncons_term}
\end{equation} 
which can be derived from the path-conservative property, using a linear path and satisfies the consistency conditions \eqref{eq:numerical_noncons_consistency_cond2}. The task now reduces to finding an accompanying entropy conservative (EC) numerical flux $\bigstatevec{f}^{\,EC}$.

To derive such EC fluxes we require discrete entropy conservation in a finite volume context similar to \cite{winters2016affordable,derigs2018ideal}. These finite volume type numerical fluxes are then directly applicable to the split-form DGSEM \eqref{eq:split_form_flux} as the volume contribution relates to a subcell finite volume scheme and elements are coupled via numerical fluxes on the surface \cite{fisher2013discretely,gassner2016split}.
We consider a single interface between neighboring finite volume cells with distinct cell averages denoted by $``-"$ and $``+"$ and cell size $\Delta x$. We introduce a notation to define jumps and arithmetic averages of cell values across the interface
\begin{equation}
	\jump{\cdot} = (\cdot)^+ -(\cdot)^-, \quad \avg{\cdot} = \frac{1}{2}\left((\cdot)^+ + (\cdot)^-\right).
\end{equation}

The respective finite volume formulation of system \eqref{eq:balance_law_continuous} with the numerical nonconservative term \eqref{eq:definition_numerical_noncons_term} is given by
\begin{equation}
	\Delta x \frac{\partial}{\partial t}\statevec{u}^- = \statevec{f}^- - \statevec{f}^* - \statevecGreek{\phi}^-\circ\jump{\statevec{r}}
	\qquad
	\Delta x \frac{\partial}{\partial t}\statevec{u}^+ = \statevec{f}^*  - \statevec{f}^+ -  \statevecGreek{\phi}^+\circ\jump{\statevec{r}},
	\label{eq:fv_formulation_condition_1}
\end{equation}
at each side of the interface. To obtain an entropy conservation statement we follow the continuous analysis and contract \eqref{eq:fv_formulation_condition_1} with the entropy variables. As before, we assume time continuity such that the chain rule in time holds to obtain the rate of entropy change in each cell
\begin{equation}
	\Delta x \frac{\partial}{\partial t}\statevec{S}^- = (\statevec{w}^-)^T \left(\statevec{f}^- - \statevec{f}^* - \frac{\statevecGreek{\phi}^-\circ\jump{\statevec{r}}}{2}\right),
	\quad
	\Delta x \frac{\partial}{\partial t}\statevec{S}^+ = (\statevec{w}^+)^T\left(\statevec{f}^* - \statevec{f}^+ - \frac{\statevecGreek{\phi}^+\circ\jump{\statevec{r}}}{2}\right).
	\label{eq:fv_formulation_condition_1_contracted}
\end{equation}
We then sum over both cells to find the total discrete entropy change over the interface
\begin{equation}
	\Delta x \frac{\partial}{\partial t}\left(\statevec{S}^-+\statevec{S}^+\right) = \jump{\statevec{w}}^T\statevec{f}^* - \jump{\statevec{w}^T\statevec{f}} - \avg{\statevec{w}\circ\statevecGreek{\phi}}^T\jump{\statevec{r}}.
	\label{eq:fv_total_discrete_entropy}
\end{equation}
From \eqref{eq:fv_total_discrete_entropy} we see that a discrete version of the integral entropy conservation law \eqref{eq:entropy_conservation_law_integral_form} holds provided the right hand side recovers the jump of the entropy flux over the interface
\begin{equation}
	\jump{\statevec{w}}^T\statevec{f}^* - \jump{\statevec{w}^T\statevec{f}} - \avg{\statevec{w}\circ\statevecGreek{\phi}}^T\jump{\statevec{r}} \overset{!}{=} 
	- \jump{\spacevec{f}^{\,S}}.
	\label{eq:entropy_conservation_requirement}
\end{equation}
Manipulating \eqref{eq:entropy_conservation_requirement} and subsequently substituting the definition of the entropy potential \eqref{eq:entropy_flux_potential} we find an entropy conservation condition for the numerical flux
\begin{equation}
	\jump{\statevec{w}}^T\bigstatevec{f}^{\,EC} =  \jump{\spacevec{\psi}} - \avg{\statevec{w}\circ\statevecGreek{\phi}}^T\jump{\bigstatevec{r}}. 
	\label{eq:entropy_conservation_condition}
\end{equation}

Accordingly, a numerical flux that satisfies this condition in combination with the numerical nonconservative term \eqref{eq:definition_numerical_noncons_term} recovers the integral entropy conservation law \eqref{eq:entropy_conservation_law_integral_form} and is therefore defined as an EC numerical flux.

With this in mind we create an EC split-form DGSEM by selecting a proper combination of EC fluxes and nonconservative terms for both volume and surface contributions. In general the split-form offers flexibility in choosing different discretizations on the volume and the surface, which was used in \cite{wintermeyer2017entropy} to obtain a well-balanced and EC scheme for the standard SWE. 

However, we find that with the numerical nonconservative term \eqref{eq:definition_numerical_noncons_term} both properties can be achieved using the same numerical nonconservative term and EC flux for the volume and surface parts in \eqref{eq:split_formulation_entropy_space}. The numerical nonconservative term \eqref{eq:definition_numerical_noncons_term} can also be applied in the volume as it corresponds to the structure given in \eqref{eq:numerical_noncons_structure} when setting
\begin{equation}
	\statevecGreek{\Phi}^{\#}(\statevec{U}^{-},\statevec{U}^{+}) = \statevecGreek{\Phi}^-, \quad \bigstatevec{R}^{\#}(\statevec{U}^{-},\statevec{U}^{+}) = \bigstatevec{R}^+
\end{equation}
and satisfies the consistency conditions \eqref{eq:numerical_noncons_consistency_cond2}. This circumvents the need for different surface and volume fluxes, such that our method is given by
\begin{equation}
	\begin{aligned}
		\iprodN{\mathbb{I}^N(J)\statevec{U}_t,\statevec{W}} + 
		&\int\limits_{\partial E,N} \statevec{W}^T \{ \statevec{F}_n^{EC} - \statevec{F}_n\} \hat{s} \dS + \iprodN{\vec{\mathbb{D}}\cdot\bigcontravec{F}^{EC} ,\statevec{W}} 
		\\+
		&\int\limits_{\partial E,N} \statevec{W}^T \left(\statevecGreek{\Phi}\circ\bigstatevec{R}\right)^{\diamond}_n \hat{s} \,\text{d}S
		+
		\iprodN{\spacevec{\mathbb{D}}\cdot(\statevecGreek{\Phi}\circ\bigcontravec{R})^{\diamond},\statevec{W}}
		= 0.
	\end{aligned}
	\label{eq:split_formulation_ec_entropy_space}
\end{equation}

We then set the following EC numerical fluxes
\begin{equation}
	\begin{aligned}
		\statevec{F}_1^{EC} := 
		\begin{pmatrix}
			\avg{h_1u_1}_{(i,m),j} \\
			\avg{h_1u_1}_{(i,m),j}\avg{u_1}_{(i,m),j} + g\avg{h_1}^2_{(i,m),j} - \frac{1}{2}g\avg{h_1^2}_{(i,m),j}\\
			\avg{h_1u_1}_{(i,m),j}\avg{v_1}_{(i,m),j}\\
			\avg{h_2u_2}_{(i,m),j} \\
			\avg{h_2u_2}_{(i,m),j}\avg{u_2}_{(i,m),j} + g\avg{h_2}^2_{(i,m),j} - \frac{1}{2}g\avg{h_2^2}_{(i,m),j}\\
			\avg{h_2u_2}_{(i,m),j}\avg{v_2}_{(i,m),j}
		\end{pmatrix}, \\
		\statevec{F}_2^{EC} := 
		\begin{pmatrix}
			\avg{h_1v_1}_{(i,m),j} \\
			\avg{h_1v_1}_{(i,m),j}\avg{u_1}_{(i,m),j} \\
			\avg{h_1v_1}_{(i,m),j}\avg{v_1}_{(i,m),j} + 	g\avg{h_1}^2_{(i,m),j} - \frac{1}{2}g\avg{h_1^2}_{(i,m),j}\\
			\avg{h_2v_2}_{(i,m),j} \\
			\avg{h_2v_2}_{(i,m),j}\avg{u_2}_{(i,m),j}\\
			\avg{h_2v_2}_{(i,m),j}\avg{v_2}_{(i,m),j} + 	g\avg{h_2}^2_{(i,m),j} - \frac{1}{2}g\avg{h_2^2}_{(i,m),j}
		\end{pmatrix},
	\end{aligned}
	\label{eq:volume_flux_ec}
\end{equation}
which fulfill the entropy conservation condition \eqref{eq:entropy_conservation_condition}
\begin{equation*}
	\jump{\statevec{W}}^T \bigstatevec{F}^{EC} = \jump{\spacevec{\Psi}}-\avg{\statevec{W}\circ\statevecGreek{\Phi}}^T\jump{\bigstatevec{R}},
\end{equation*}
as shown in Appendix \ref{sec:entropy_condition_volume} for $\statevec{F}_1^{EC}$. This EC flux \eqref{eq:volume_flux_ec} corresponds to the EC volume flux found by Wintermeyer et al. \cite{wintermeyer2017entropy} for the standard SWE within each layer. The reason behind this is that the conservative part of the two-layer model \eqref{eq:two_layer_system_expanded} corresponds to the standard shallow water system in each layer and only the nonconservative terms differ.
\begin{remark}
	We note that due to this similarity the EC split-form DGSEM \eqref{eq:split_formulation_ec_entropy_space} can be directly applied to the standard SWE if the EC flux from \cite{wintermeyer2017entropy} and the numerical nonconservative term \eqref{eq:definition_numerical_noncons_term} are used.
\end{remark}

To demonstrate entropy conservation for approximation \eqref{eq:split_formulation_ec_entropy_space} we next examine the volume and surface contributions separately. We start with the volume parts and show that for the given choice of two-point volume flux \eqref{eq:volume_flux_ec} and numerical nonconservative term \eqref{eq:definition_numerical_noncons_term} they become the entropy flux evaluated at the boundary, when contracted into entropy space. Using the same combination of EC flux and numerical nonconservative term the total surface contribution then cancels in entropy space, which shows entropy conservation. Furthermore, we demonstrate that by using an analogous discretization for the pressure and nonconservative terms yields a method that is well-balanced.

\subsection{Volume contribution}

\begin{lemma}[Entropy contribution of the curvilinear volume terms]\label{lemma_volume_contribution}
	The curvilinear volume contributions of the split-form DGSEM \eqref{eq:split_formulation} with the EC volume flux \eqref{eq:volume_flux_ec} contracted into entropy space become the entropy flux evaluated at the interface.
	\begin{equation}
		\begin{aligned}
			\iprodN{\vec{\mathbb{D}}\cdot\bigcontravec{F}^{EC} ,\statevec{W}} + \iprodN{\spacevec{\mathbb{D}}\cdot(\statevecGreek{\Phi}\circ\bigcontravec{R})^{\diamond},\statevec{W}} = 		\int\limits_{\partial E,N} \left(\spacevec{F}^{\,S}\cdot\spacevec{n}\right) \hat{s} \dS
		\end{aligned}
	\end{equation}
\end{lemma}
\begin{proof}
	See Appendix \ref{sec:proof_volume_contribution_entropy}. 
\end{proof}
\begin{remark}
	The proof requires the SBP property \eqref{eq:SBP_integration_by_parts}, the entropy conservation condition \eqref{eq:entropy_conservation_condition}, and that the metric identities \eqref{eq:metric_identity} are satisfied discretely.
\end{remark}

The result from \mbox{Lemma \ref{lemma_volume_contribution}} shows that the volume contributions from the flux and nonconservative terms cancel in entropy space and become the entropy flux evaluated at the boundary. Therefore, the volume parts do not contribute to the total discrete entropy and the entropy balance is entirely dependent on the surface contributions.

\subsection{Surface contribution}
On the surface we select the same EC numerical flux as in the volume. From Lemma \ref{lemma_volume_contribution} we know that the volume contributions move onto the surface such that the change in discrete entropy depends solely on the surface terms. So, we assemble the surface contributions for a single element 
\begin{equation}
		\begin{aligned}
			\Gamma_k =
			&\int\limits_{\partial E_k,N} \statevec{W}^T \{ \statevec{F}_n^{EC} - \statevec{F}_n\} \hat{s} \dS + \int\limits_{\partial E_k,N}\left( \spacevec{F}^{\,S}\cdot\spacevec{n}\right)\hat{s}\dS\\
			+ 
			&\int\limits_{\partial E,N} \statevec{W}^T \left(\statevecGreek{\Phi}\circ\bigstatevec{R}\right)^{\diamond}_n \hat{s} \,\text{d}S , \hspace{3cm} k = 1,...,K.
		\end{aligned}
		\label{eq:surface_contribution_element}
\end{equation}

To obtain the total discrete entropy as described by \eqref{eq:total_discrete_entropy} we sum over all elements $k=1,...,K$. Due to the discontinuous approximation, summing over the elements creates jumps in the fluxes, nonconservative terms and entropy variables, whereas the numerical surface fluxes are defined uniquely at each interface. 

After summing over the elements, we examine the terms in the surface integrals \eqref{eq:surface_contribution_element} separately. Starting with the physical flux term we obtain jumps in the flux and entropy variables and use that the numerical flux is unique at the interface to have
\begin{equation}
	\sum\limits_{k=1}^{K}\int\limits_{\partial E_k,N} \statevec{W}^T \{ \statevec{F}_n^{EC} - \statevec{F}_n\} \hat{s} \dS
	= 
	\sum_{\text{faces}} \int\limits_{N}
	\left\{
	\jump{\statevec{W}}^T \bigstatevec{F}^{EC} - \jump{\statevec{W}^T \bigstatevec{F}} 
	\right\} \cdot\spacevec{n}\hat{s} \dS.
	\label{eq:surface_sum_physical_flux}
\end{equation}
In a similar way summing over the second integral generates a jump in the entropy flux
\begin{equation}
	\sum\limits_{k=1}^{K}\int\limits_{\partial E_k,N} \left(\spacevec{F}^{\,S}\cdot\spacevec{n}\right) \hat{s} \dS
	=
	\sum_{\text{faces}} \int\limits_{N} \jump{\spacevec{F}^{\,S}}\cdot\spacevec{n} \hat{s} \dS.
	\label{eq:surface_sum_entropy_flux}
\end{equation}
The third integral in \eqref{eq:surface_contribution_element} contains the the numerical nonconservative term \eqref{eq:definition_numerical_noncons_term}, which is non-unique at the interface. In the following we examine the total surface contribution of the nonconservative parts after summing over the interfaces.

\begin{lemma}[Surface contribution of the numerical nonconservative term]\label{lemma:diamond_flux}
	Setting the numerical nonconservative term to \eqref{eq:definition_numerical_noncons_term}
	\begin{equation*}
		\left(\statevecGreek{\Phi}\circ\bigstatevec{R}\right)^{\diamond} = \frac{1}{2}\statevecGreek{\Phi}^-\circ\jump{\bigstatevec{R}}
	\end{equation*}
	yields the total surface contribution of the nonconservative terms
	\begin{equation}
		\sum\limits_{k=1}^{K}\int\limits_{\partial E_k,N} \statevec{W}^T
		\left(\statevecGreek{\Phi}\circ\bigstatevec{R}\right)_n^{\diamond} 
		\hat{s} \dS 
		= 
		\sum_{\text{faces}} \int\limits_{N}
		\avg{\statevec{W}\circ\statevecGreek{\Phi}}^T\jump{\bigstatevec{R}}\cdot\spacevec{n}\hat{s} \dS.
	\end{equation}
\end{lemma}
\begin{proof}
	The total surface contribution from the nonconservative terms in \eqref{eq:surface_contribution_element} is given by
	\begin{equation}
		\sum\limits_{k=1}^{K}\int\limits_{\partial E_k,N} \statevec{W}^T
		(\statevecGreek{\Phi}\circ\bigstatevec{R})_n^{\diamond} 
		\hat{s} \dS.
		\label{eq:surface_contribution_noncons}
	\end{equation}
	Introducing the definition for the numerical nonconservative term \eqref{eq:definition_numerical_noncons_term} into \eqref{eq:surface_contribution_element} and examining the nonconservative contribution, where \mbox{$``-"$} denotes the primary and $``+"$ the secondary state, yields
	\begin{equation}
		\sum\limits_{k=1}^{K}\int\limits_{\partial E_k,N} \statevec{W}^T
		(\statevecGreek{\Phi}\circ\bigstatevec{R})_n^{\diamond} 
		\hat{s} \dS
		=
		\sum\limits_{k=1}^{K}\int\limits_{\partial E_k,N} \statevec{W}^T
		\left(
		\frac{1}{2}\statevecGreek{\Phi}^-\circ\jump{\bigstatevec{R}}\right)\cdot\spacevec{n}
		\hat{s} \dS.
	\end{equation}

	To obtain the total interface contribution, we must take contributions from neighboring cells into account. Since the nonconservative term is not symmetric, we need to consider the different normal directions of neighboring cells as defined in \eqref{eq:defintion_normal_vector_pm}.  Summing over all elements, we then obtain the total contribution from the interfaces
	\begin{equation}
		\begin{aligned}
			&\sum\limits_{k=1}^{K}\int\limits_{\partial E_k,N} \statevec{W}^T
			(\statevecGreek{\Phi}\circ\bigstatevec{R})_n^{\diamond}\hat{s}\dS
			\\=
			&\sum\limits_{faces}\int\limits_{N}\frac{1}{2} \left(\left(\statevec{W}^-\right)^T
			\left(
			\statevecGreek{\Phi}^-\circ\jump{\bigstatevec{R}}\right)
			+
			\left(\statevec{W}^+\right)^T
			\left(
			\statevecGreek{\Phi}^+\circ\jump{\bigstatevec{R}}\right)
			\right)
			\cdot\spacevec{n}\hat{s} \dS
			\\=
			&\sum\limits_{faces}\int\limits_{N} \avg{\statevec{W}\circ\statevecGreek{\Phi}}^T
			\jump{\bigstatevec{R}}\cdot\spacevec{n}\hat{s}\dS,
		\end{aligned}
	\end{equation}
	where we used that due to anti-symmetry the different sign in the jump operator between $``+"$ and $``-"$ cancels with the sign flip from opposite normal directions.
\end{proof}

\begin{lemma}[Total entropy contribution of the surface terms]\label{corollary_entropy_contribution_surface}
	When summing over all elements, the total entropy contributions of the advective and nonconservative terms in \eqref{eq:surface_contribution_element} cancel on the surface 
	\begin{equation}
		\sum_{k=1}^K\int\limits_{\partial E_k,N} \left(\statevec{W}^T \{ \statevec{F}_n^{EC} - \statevec{F}_n\} + \left(\spacevec{F}^{\,S}\cdot\spacevec{n}\right)
		+ \statevec{W}^T
		(\statevecGreek{\Phi}\circ\statevec{R})_n^{\diamond} 
		\right)
		\hat{s} \dS = 0.
	\end{equation}
\end{lemma}
\begin{proof}
	First we write the fluxes in terms of jumps according to \eqref{eq:surface_sum_physical_flux}, \eqref{eq:surface_sum_entropy_flux} and apply the result of Lemma \ref{lemma:diamond_flux} for the nonconservative contributions to obtain
	\begin{equation}
		\sum_{\text{faces}} \int\limits_{N}
		\left\{
		\jump{\statevec{W}}^T \bigstatevec{F}^{EC} - \jump{\statevec{W}^T \bigstatevec{F}} 
		+
		\jump{\spacevec{F}^{\,S}}
		+
		\avg{\statevec{W}\circ\statevecGreek{\Phi}}^T\jump{\bigstatevec{R}}\right\}\cdot\spacevec{n}\hat{s} \dS.
	\end{equation}	

	Now we substitute the definition of the entropy flux potential \eqref{eq:entropy_flux_potential} 	
  	\begin{equation}
	  	\sum_{\text{faces}} \int\limits_{N}
	  	\left\{
	  	\jump{\statevec{W}}^T \bigstatevec{F}^{EC} - 	\jump{\spacevec{\Psi}} 
	  	+
		\avg{\statevec{W}\circ\statevecGreek{\Phi}}^T\jump{\bigstatevec{R}}\right\}\cdot\spacevec{n}\hat{s}\dS,
	\end{equation}	
	and substitute the entropy conservation condition \eqref{eq:entropy_conservation_condition} to cancel the remaining terms and therefore show that the total entropy contribution on the surface is zero. 
\end{proof}

Now that the surface and volume contributions are examined, we have everything assembled and combine the previous results to show that approximation \eqref{eq:split_formulation} is EC.

\begin{theorem}[Entropy conservation of the curvilinear split-form DGSEM for the two-layer shallow water equations]\label{theorem_ec}	
The curvilinear split-form DGSEM for the two-layer shallow water equations
\begin{equation}
	\begin{aligned}
		\iprodN{\mathbb{I}^N(J)\statevec{U}_t,\statevecGreek{\varphi}} + 
		&\int\limits_{\partial E,N} \statevecGreek{\varphi}^T \{ \statevec{F}_n^* - \statevec{F}_n\} \hat{s} \dS + \iprodN{\vec{\mathbb{D}}\cdot\bigcontravec{F}^{\#} ,\statevecGreek{\varphi}} 
		\\+
		&\int\limits_{\partial E,N} \statevecGreek{\varphi}^T
		(\statevecGreek{\Phi}\circ\statevec{R})_n^{\diamond} 
		\hat{s} \dS 
		+
		\iprodN{\spacevec{\mathbb{D}}\cdot(\statevecGreek{\Phi}\circ\bigcontravec{R})^{\#},\statevecGreek{\varphi}}
		= 0
	\end{aligned}
	\label{eq:ec_discretization}
\end{equation}
with the EC flux \eqref{eq:volume_flux_ec} and the numerical nonconservative term \eqref{eq:definition_numerical_noncons_term}
\begin{equation}
	\statevec{F}^{\#} = \statevec{F}^* = \bigstatevec{F}^{EC}, \quad
	(\statevecGreek{\Phi}\circ\statevec{R})^{\#} =
	(\statevecGreek{\Phi}\circ\statevec{R})^{\diamond} =
	\frac{1}{2}\statevecGreek{\Phi}^- \circ\jump{\bigstatevec{R}},
	\label{eq:ec_fluxes}
\end{equation}
is entropy conservative.
\end{theorem}
\begin{proof}
	From Lemma \ref{lemma_volume_contribution}, we have that the entropy contributions cancel in the volume and generate the entropy flux on the surface. Then we apply the result from \mbox{Lemma \ref{corollary_entropy_contribution_surface}} to show that the entropy contribution cancels on the surface and therefore entropy is conserved.
\end{proof}

\subsection{Well-Balancedness}\label{sec:well_balanced}
In the following we show that in addition to entropy conservation, the discretization \eqref{eq:ec_discretization} is well-balanced as it satisfies the lake-at-rest condition \eqref{eq:lake_at_rest_condition}.

\begin{corollary}[Well-balancedness of the curvilinear split-form DGSEM for the two-layer shallow water equations]\label{corollary_well_balanced}
	The EC curvilinear split-form DGSEM for the two-layer shallow water equations
	\begin{equation}
		\begin{aligned}
			\iprodN{\mathbb{I}^N(J)\statevec{U}_t,\statevecGreek{\varphi}} + 
			&\int\limits_{\partial E,N} \statevecGreek{\varphi}^T \{ \statevec{F}_n^{EC} - \statevec{F}_n\} \hat{s} \dS + \iprodN{\vec{\mathbb{D}}\cdot\bigcontravec{F}^{EC} ,\statevecGreek{\varphi}} 
			\\+
			&\int\limits_{\partial E,N} \statevecGreek{\varphi}^T
			(\statevecGreek{\Phi}\circ\statevec{R})_n^{\diamond} 
			\hat{s} \dS 
			+
			\iprodN{\spacevec{\mathbb{D}}\cdot(\statevecGreek{\Phi}\circ\bigcontravec{R})^{\diamond},\statevecGreek{\varphi}}
			= 0
		\end{aligned}
	\end{equation}
	 preserves the lake-at-rest initial condition \eqref{eq:lake_at_rest_condition}
	 \begin{equation*}
	 	u_1,u_2,v_1,v_2 = 0, \quad b+h_2 = const, \quad h_1 = const.
	 \end{equation*}
\end{corollary}
\begin{proof}
	To show that the scheme is well-balanced, we take the same approach as demonstrating entropy conservation. That is we examine the volume and surface parts separately, to see that both contributions vanish for the lake-at-rest condition \eqref{eq:lake_at_rest_condition}.\\
	In the continuity equations in \eqref{eq:two_layer_system_fluxes}, it is clear that well-balancedness is directly satisfied as all products vanish due to the initial condition $u_1,u_2,v_1,v_2=0$. Therefore, only the momentum equations need to be considered to show well-balancedness. For simplicity the proof is only demonstrated for the $h_1u_1$-equation as the other components are analogous.
	
	Inserting the lake-at-rest conditions \eqref{eq:lake_at_rest_condition} into the EC approximation \eqref{eq:ec_discretization} and using that $\smash{\avg{h}^2 - \frac{1}{2}\avg{h^2} = \frac{h^+h^-}{2}}$, the volume flux contribution simplifies to
\begin{equation}
		\begin{aligned}
			\iprodN{\vec{\mathbb{D}}\cdot\bigcontravec{F}^{EC} ,\statevecGreek{\varphi_{h_1u_1}}} = 
			\sum\limits_{i,j=0}^N \omega_{ij}
			&\left[2\sum_{m=0}^{N} \mathcal{D}_{im} \left(\left( g\avg{h_1}_{(i,m)j}^2 - \frac{1}{2}g\avg{h_1^2}_{(i,m)j}\right) \avg{J\vec{a}_1^1}_{(i,m)j} \right)\right.
			\\&\hspace{-0.18cm}+
			\left.
			2\sum_{m=0}^{N} \mathcal{D}_{jm}   \left(\left( g\avg{h_1}_{i(j,m)}^2 - \frac{1}{2}g\avg{h_1^2}_{i(j,m)}\right) \avg{J\vec{a}_1^2}_{i(j,m)} \right)\right]
			\\=
			g\sum\limits_{i,j=0}^N \omega_{ij} \left(h_1\right)_{ij}&\left[\sum_{m=0}^{N} \mathcal{D}_{im} \left(\left(h_1\right)_{mj} \avg{J\vec{a}_1^1}_{(i,m)j} \right)\right.
			\\&\hspace{-0.18cm}+
			\left.
			\sum_{m=0}^{N} \mathcal{D}_{jm}   \left(\left(h_1\right)_{im} \avg{J\vec{a}_1^2}_{i(j,m)} \right)\right].
		\end{aligned}
		\label{eq:wb_volume_contribution_flux}
	\end{equation}	
The nonconservative volume contribution is then rewritten into a similar form. Therefore, we first use that the metric identities \eqref{eq:metric_identity} are satisfied such that any terms containing only local contributions $\smash{\statevecGreek{\Phi}_{ij}\circ\bigstatevec{R}_{ij}
}$ vanish to obtain
\begin{equation}
	\begin{aligned}
		\statevecGreek{\Phi}\circ \left(\spacevec{\nabla}_{\xi}\cdot\mathbb{I}^N(\bigcontravec{R})\right)
		\approx
		\spacevec{\mathbb{D}}\cdot(\statevecGreek{\Phi}\circ\bigcontravec{R})^{\diamond} = 
		&\sum_{m=0}^{N} \mathcal{D}_{im} \left(\statevecGreek{\Phi}_{ij}\circ\left(\bigstatevec{R}_{mj} \cdot \avg{J\vec{a}^{\,1}}_{(i,m)j} \right)\right)
		\\ +
		&\sum_{m=0}^{N} \mathcal{D}_{jm}   \left(\statevecGreek{\Phi}_{ij}\circ\left(\bigstatevec{R}_{im} \cdot \avg{J\vec{a}^{\,2}}_{i(j,m)} \right)\right),
	\end{aligned}
\end{equation}
By expanding the nonconservative term, the contribution to the $h_1u_1$-equation is then given by
\begin{equation}
	\begin{aligned}
		\iprodN{\spacevec{\mathbb{D}}\cdot(\statevecGreek{\Phi}\circ\bigcontravec{R})^{\diamond},\statevecGreek{\varphi_{h_1u_1}}}
		=
		g\sum\limits_{i,j=0}^N \omega_{ij} \left(h_1\right)_{ij}
		&\left[\sum_{m=0}^{N} \mathcal{D}_{im} \left(\left(b+h_2\right)_{mj} \avg{J\vec{a}_1^1}_{(i,m)j} \right) \right.
		\\ &\hspace{-0.18cm}+\left.
		\sum_{m=0}^{N} \mathcal{D}_{jm}   \left(\left(b+h_2\right)_{im} \avg{J\vec{a}_1^2}_{i(j,m)} \right)\right].
	\end{aligned}
	\label{eq:wb_volume_contribution_nonconservative}
\end{equation}
Next we assemble the total volume contribution from \eqref{eq:wb_volume_contribution_flux} and \eqref{eq:wb_volume_contribution_nonconservative}. We first use that $b+h_1+h_2=\text{constant}$, according to \eqref{eq:lake_at_rest_condition}. Then, from the consistency of the derivative matrix and assuming that the metric identities \eqref{eq:metric_identity} hold discretely, the volume parts vanish for the lake-at-rest condition
\begin{equation}
	\begin{aligned}
		\iprodN{\vec{\mathbb{D}}\cdot\bigcontravec{F}^{EC} ,\statevecGreek{\varphi_{h_1u_1}}} + &\iprodN{\spacevec{\mathbb{D}}\cdot(\statevecGreek{\Phi}\circ\bigcontravec{R})^{\diamond},\statevecGreek{\varphi_{h_1u_1}}}
		\\&\hspace{-2cm}=
		\sum\limits_{i,j=0}^N \omega_{ij} g\left(h_1\right)_{ij}
		&&\hspace{-3.1cm}\left[\sum_{m=0}^{N} \mathcal{D}_{im} \left(\left(b+h_1+h_2\right)_{mj} \avg{J\vec{a}_1^1}_{(i,m)j} \right) \right.
		\\& &&\hspace{-3.1cm}\hspace{-0.18cm}+\left.
		\sum_{m=0}^{N} \mathcal{D}_{jm}   \left(\left(b+h_1+h_2\right)_{im} \avg{J\vec{a}_1^2}_{i(j,m)} \right)\right]
		\\&\hspace{-2cm}=
		\frac{g}{2}\left(b+h_1+h_2\right)\sum\limits_{i,j=0}^N \omega_{ij} \left(h_1\right)_{ij}
		&&\hspace{-1.1cm}\left[\sum_{m=0}^{N} \mathcal{D}_{im} \left(\left(J\vec{a}_1^1\right)_{ij} +\left(J\vec{a}_1^1\right)_{mj} \right)   \right.
		\\& &&\hspace{-1.1cm}\hspace{-0.18cm}+ \left.
		\sum_{m=0}^{N} \mathcal{D}_{jm} \left(\left(J\vec{a}_1^1\right)_{ij} +\left(J\vec{a}_1^1\right)_{im} \right)\right]
		\\&\hspace{-2cm}=
		\frac{g}{2}\left(b+h_1+h_2\right)\sum\limits_{i,j=0}^N \omega_{ij} \left(h_1\right)_{ij}
		&&\hspace{-1.1cm}\left[\sum_{m=0}^{N} \mathcal{D}_{im} \left(J\vec{a}_1^1\right)_{mj} + 
		\sum_{m=0}^{N} \mathcal{D}_{jm} \left(J\vec{a}_1^1\right)_{im} \right]
		\\&\hspace{-2cm}=
		\,0.
	\end{aligned}
\end{equation}

The remaining surface contributions are given by
\begin{equation}
	\int\limits_{\partial E,N} \statevecGreek{\varphi}_{h_1u_1}^T \{ \statevec{F}_n^* - \statevec{F}_n\} \hat{s} \dS +
	\int\limits_{\partial E,N} \statevecGreek{\varphi}_{h_1u_1}^T
	(\statevecGreek{\Phi}\circ\statevec{R})_n^{\diamond} 
	\hat{s} \dS.
	\label{eq:wb_surface_contribution}
\end{equation}

Expanding the terms in the $h_1u_1$-equation, again all products with velocity components vanish, and \eqref{eq:wb_surface_contribution} simplifies to 
\begin{equation}
	\begin{aligned}
		&\int\limits_{\partial E,N} \left(g\avg{h_1}^2 -\frac{1}{2}g\avg{h_1^2} - \frac{1}{2}g\left(h_1^-\right)^2 
		+ 
		\frac{1}{2}gh_1^-\jump{b+h_2}\right)n_1\hat{s}
		\dS
		\\=
		&\frac{g}{2}\int\limits_{\partial E,N} \left(2\avg{h_1}^2 - \avg{h_1^2} - \left(h_1^-\right)^2 
		+ 
		h_1^-\jump{b+h_2}\right)n_1\hat{s}
		\dS.
		\label{eq:wb_surface_step_1}
	\end{aligned}
\end{equation}

The first three terms are rewritten as
\begin{equation}
	2\avg{h_1}^2 - \avg{h_1^2} - \left(h_1^-\right)^2 = h_1^-h_1^+ - (h_1^-)^2 = h_1^-\jump{h_1}.
\end{equation}	
Inserting this into \eqref{eq:wb_surface_step_1} and using that $\jump{b+h_1+h_2}=0$ according to the lake-at-rest condition yields

\begin{equation}
	\begin{aligned}
		&\frac{g}{2}\int\limits_{\partial E,N}
		\left(
		{h_1^-}\jump{h_1} 
		+ {h_1^-}\jump{b+h_2}\right)n_1\hat{s}\dS
		\\&=
		\frac{g}{2}\int\limits_{\partial E,N} \left({h_1^-}\jump{b+h_1+h_2}
		\right)n_1\hat{s}\dS 
		\\
		&= 0,
	\end{aligned}
\end{equation}
which shows that the $h_1u_1$-equation satisfies \eqref{eq:lake_at_rest_condition} as both the volume and the surface contributions vanish. As the remaining momentum equations are analogous, this shows that the EC split-form DGSEM approximation \eqref{eq:ec_discretization} is well-balanced for the two-layer SWE. 
\end{proof}

\section{Entropy-Stable DGSEM}\label{sec:entropy_stable_dgsem}
The high-order approximation introduced in Theorem \ref{theorem_ec} satisfies the entropy conservation statement \eqref{eq:entropy_conservation_law} in a semi-discrete setting. In the presence of discontinuities entropy conservation is not sufficient and instead the entropy inequality \eqref{eq:entropy_inequality} must be satisfied to account for entropy dissipation at shocks. Following the idea of Tadmor \cite{tadmor1987numerical}, we add numerical dissipation to the EC scheme from Theorem \ref{theorem_ec} to construct a discretization that is entropy stable (ES). We use the approach from \cite{fjordholm2012energy} and create an ES flux by adding a local Lax-Friedrichs type dissipation to the EC flux

\begin{equation}
	\bigstatevec{F}^{ES} = \bigstatevec{F}^{EC} - \frac{1}{2}|\lambda_{max}|\bar{\statevec{H}}\jump{\statevec{w}}, 
	\label{eq:flux_es}
\end{equation}
where $\lambda_{max}$ denotes an approximation of the maximum eigenvalue from \eqref{eq:bound_eigenvalue} and $\bar{\statevec{H}} = \statevec{w}_{\statevec{u}}$ is the symmetric positive definite entropy Jacobian matrix, evaluated at the arithmetic average value of the solution state at element interfaces. That the dissipation is dependent on the jump in entropy variables guarantees that entropy is dissipated at shocks and conserved for smooth solutions. 

\begin{theorem}[Entropy stability of the curvilinear split-form DGSEM for the two-layer shallow water equations]
The curvilinear split-form DGSEM for the two-layer shallow water equations
	\begin{equation}
		\begin{aligned}
			\iprodN{\mathbb{I}^N(J)\statevec{U}_t,\statevecGreek{\varphi}} + 
			&\int\limits_{\partial E,N} \statevecGreek{\varphi}^T \{ \statevec{F}_n^{ES} - \statevec{F}_n\} \hat{s} \dS + \iprodN{\vec{\mathbb{D}}\cdot\bigcontravec{F}^{EC} ,\statevecGreek{\varphi}} 
			\\+
			&\int\limits_{\partial E,N} \statevecGreek{\varphi}^T
			(\statevecGreek{\Phi}\circ\statevec{R})_n^{\diamond}
			\hat{s} \dS 
			+
			\iprodN{\spacevec{\mathbb{D}}\cdot(\statevecGreek{\Phi}\circ\bigcontravec{R})^{\diamond},\statevecGreek{\varphi}}
			= 0
		\end{aligned}
		\label{eq:es_discretization}
	\end{equation}
	is ES and well-balanced, when the ES flux \eqref{eq:flux_es} is used at element interfaces, together with the EC flux \eqref{eq:ec_fluxes} on the volume and the numerical nonconservative term \eqref{eq:definition_numerical_noncons_term}.
\end{theorem}
\begin{proof}
	Entropy stability follows, provided \eqref{eq:es_discretization} is more dissipative than the EC discretization \eqref{eq:ec_discretization}. From Lemma \ref{corollary_entropy_contribution_surface} it is clear that the surface contributions of \eqref{eq:es_discretization} vanish together with the EC flux part of \eqref{eq:flux_es} when contracted into entropy space. Therefore, the change in total entropy obtained from summing over all elements depends only on the remaining dissipation term
	\begin{equation}
		\frac{\operatorname{d\bar{S}}}{{\operatorname{dt}}} =- \sum\limits_{faces}\int\limits_{N}
		\left(\frac{1}{2}\jump{\statevec{W}}^T|\lambda_{max}|\bar{\statevec{H}}\jump{\statevec{W}}\right) \cdot \spacevec{n}\hat{s}\dS \leq 0.
	\end{equation}
	Using that $\bar{\statevec{H}}$ is the Hessian of a convex function and therefore positive definite, the total entropy decreases in time and the discretization is therefore ES.
	For the lake-at-rest initial condition \eqref{eq:lake_at_rest_condition}, the entropy variables \eqref{eq:entropy_variables} remain constant and the ES flux simplifies to the EC flux. The proof for well-balancedness then follows directly from Corollary \ref{corollary_well_balanced}. 
\end{proof}

\section{Results}\label{sec:results}
To demonstrate the theoretical findings and performance of the split-form DGSEM, we present and discuss results from several numerical experiments. First, we investigate convergence properties and show spectral convergence on a curvilinear mesh. Then we provide numerical evidence of the well-balanced property for discontinuous bottom topography and add a perturbation to this test case to demonstrate entropy stability. We then close the discussion with a more complex test, showcasing a partial dam break of a parabolic dam.

Numerical results are obtained using the open-source framework Trixi.jl \cite{ranocha2022adaptive,schlottkelakemper2021purely}. A reproducibility repository is available on Zenodo \cite{ersing_2023repo} and GitHub\footnote{\url{https://github.com/trixi-framework/paper-2023-es_two_layer}} for the results presented herein.
All computations use a low storage five-stage fourth-order Runge-Kutta scheme of Carpenter and Kennedy \cite{CarpenterKennedy94} for time integration, with either fixed or CFL-based time stepping \cite{courant1967partial}.

\subsection{Convergence Results}
We apply the method of manufactured solutions  to demonstrate spectral convergence of the numerical approximation against an exact reference solution. To construct the exact solution, layer heights and bathymetry are defined by trigonometric functions and constants are used for the velocities

\begin{equation}
	\begin{aligned}
		u_1 &= 0.9, \quad u_2 = 1, \quad v_1 = 1, \quad v_2 = 0.9\\
		b\hspace{3.5pt}   &= 1 + 0.1 \cos(\pi x) \sin(\pi y)\\
		h_1 &= 2 + 0.1 \sin(2\pi x + t) \cos(2\pi y + t)\\
		h_2 &= 4 + 0.1 \cos(2\pi x + t) \sin(2\pi y + t) - h_1 - b.
	\end{aligned}
	\label{eq:IC_convergence}
\end{equation}

The additional source terms for the manufactured solution are then computed using analytical derivatives 
\begin{equation}
	src = 
	\begin{pmatrix} h_1 \\ h_1u_1 \\ h_1v_1 \\ h_2 \\ h_2u_2 \\ h_2v_2 \end{pmatrix}_t
	+
	\begin{pmatrix}
		h_1u_1 \\ h_1u_1^2 + 0.5gh_1^2 \\ h_1u_1v_1 \\ h_2u_2 \\ h_2u_2^2 + 0.5gh_2^2 \\ h_2u_2v_2
	\end{pmatrix}_x
	+
	\begin{pmatrix}
		h_1v_1 \\ h_1u_1v_1 \\ h_1v_1^2 + 0.5gh_1^2 \\ h_2v_2 \\ h_2u_2v_2 \\ h_2v_2^2 + 0.5gh_1^2
	\end{pmatrix}_y
	+
	\begin{pmatrix}
		0 \\ h_1g\left(b+h_2\right)_x \\ h_1g\left(b+h_2\right)_y \\ 0 \\ h_2g\left(b+\frac{\rho_1}{\rho_2}h_1\right)_x \\ h_2g\left(b+\frac{\rho_1}{\rho_2}h_1\right)_y
	\end{pmatrix},
\end{equation}
\begin{figure}[htb]
	\centering
	\includegraphics[width=0.65\linewidth]{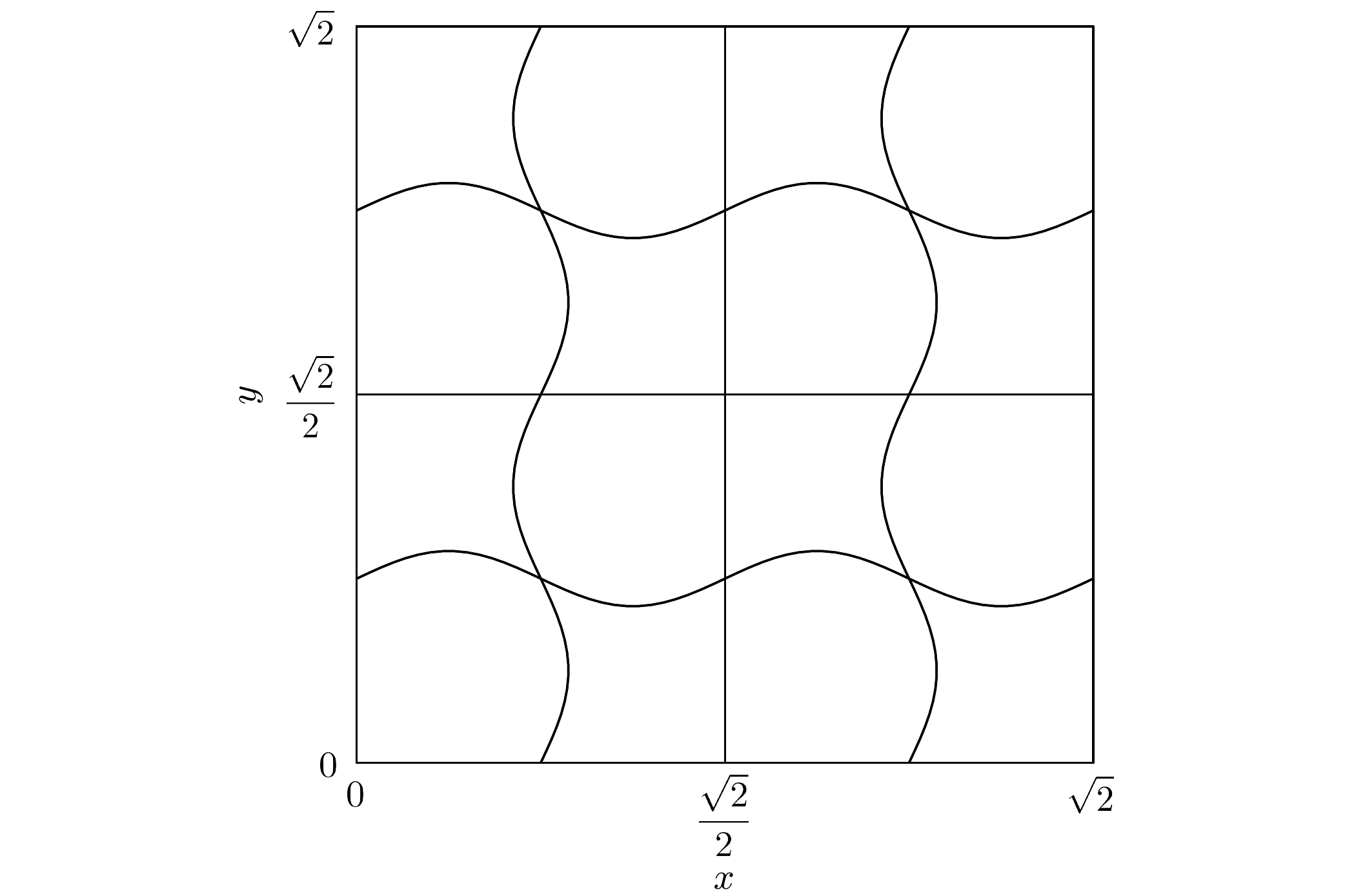}
	\caption{Curvilinear mesh adapted from \cite{wintermeyer2017entropy} with polynomial degree $N=6$ representation of interfaces used to test spectral convergence and well-balancedness}\label{fig:mesh_convergence}
\end{figure}
with the gravitational constant set to $g=10$ and the density ratio $\frac{\rho_1}{\rho_2}=0.9$.

The resulting problem is solved in the domain $\smash{[0,\sqrt{2}]^2}$ on the curvilinear mesh shown in Figure \ref{fig:mesh_convergence} using polynomial degree $N=6$ representation at element edges. Solutions are computed at the final time $t_{end}=1$ for polynomial degrees up to $N=30$ and a fixed timestep $\Delta t = 1/12000$.

Spectral convergence for both the ES and EC approximations is shown in Figure \ref{fig:convergence_plot}, which shows a semi-log plot for the error in $L_2$ over the polynomial degree $N$.
We see suboptimal convergence for the EC scheme for various polynomial degrees $N$. The characteristic pattern of suboptimal convergence at odd polynomial degree, reported in \cite{gassner2013skew,gassner2016split} cannot be observed for curved elements, but does occur under the initial conditions \eqref{eq:IC_convergence} and resolution parameters on a mesh with straight element edges. The convergence behavior of the remaining quantities was similar and is not shown.
\begin{figure}[htb]
	\centering
	\includegraphics[width=0.7\textwidth]{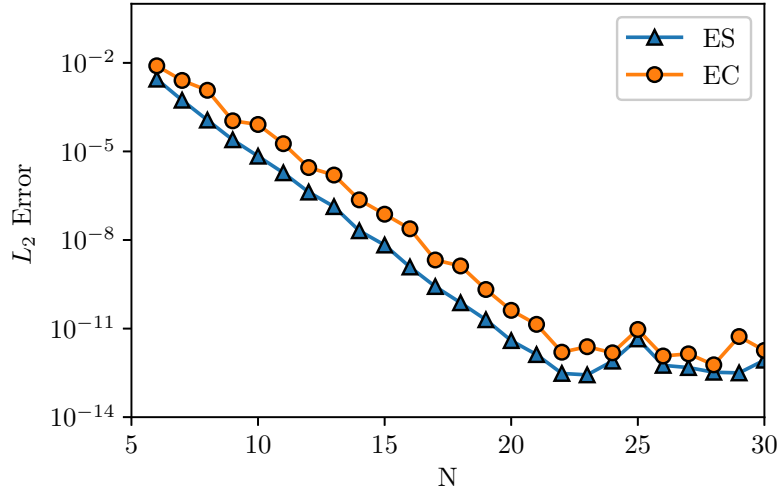}
	\caption{Spectral convergence of ES and EC approximations in space shown in a semi-log plot for the $L_2$-Error in $h_1u_1$ over the polynomial degree $N$. Results are obtained with a fixed timestep of $\Delta t = 1/12000$ at the final time $t_{end}=1$} 
	\label{fig:convergence_plot}
\end{figure}
  
\subsection{Well-balanced}\label{sec:results_well_balanced}
The next test is adapted from \cite{wintermeyer2017entropy} and demonstrates that the approximation preserves the lake-at-rest condition \eqref{eq:lake_at_rest_condition} for the two-layer SWE with discontinuous bottom topography, according to Corollary \ref{corollary_well_balanced}. Again we use the curvilinear mesh from Figure \ref{fig:mesh_convergence} and set the following lake-at-rest conditions with a discontinuous bottom topography at element $3\times2$ (see Figure \ref{fig:well_balanced})
\begin{equation}
	\begin{aligned}
		&H_1 = 0.6, \quad H_2  = 0.5, \quad u_1,u_2,v_1,v_2 = 0, \quad \frac{\rho_1}{\rho_2} = 0.9\\
		&b = \begin{cases}
			0.25 + 0.1\sin\left(2\pi x\right) + 0.1\cos\left(2\pi y\right), & \text{for element $3\,\times\,2$}\\
			0,  & \text{else}.
		\end{cases} 
	\end{aligned}
	\label{eq:well_balanced_IC}
\end{equation}

We set periodic boundary conditions and compute the solution with a polynomial degree $N=8$ and a CFL-based timestep up to the final time $t_{end}=100$. In Figure \ref{fig:well_balanced} we show the bottom topography and present the error in the absolute water height $H_1=h_1+h_2+b$ relative to the lake-at-rest initial condition \eqref{eq:well_balanced_IC}. Even after considerate time of $t=100$, errors in $H_1$ and $H_2$ remain around machine precision and the steady-state is preserved. Similar results with errors near unit roundoff were also obtained with the ES flux.
\begin{figure}
	\subfigure[bottom topography]{
		\includegraphics[width=0.49\linewidth,trim={5cm 2cm 11cm 3cm},clip]{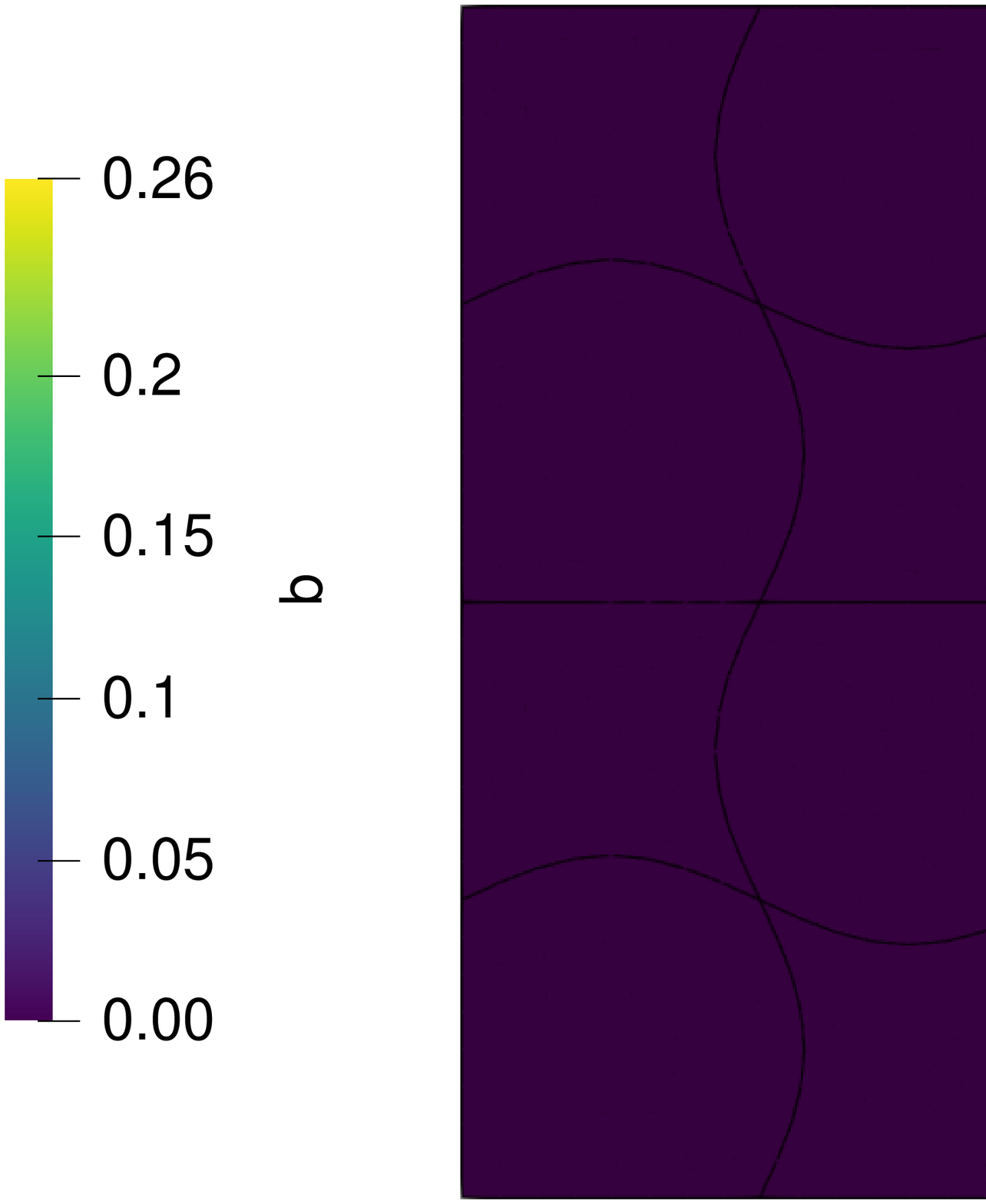}
	}
	\subfigure[lake-at-rest error]{
		\includegraphics[width=0.49\linewidth,trim={5cm 2cm 11cm 3cm},clip]{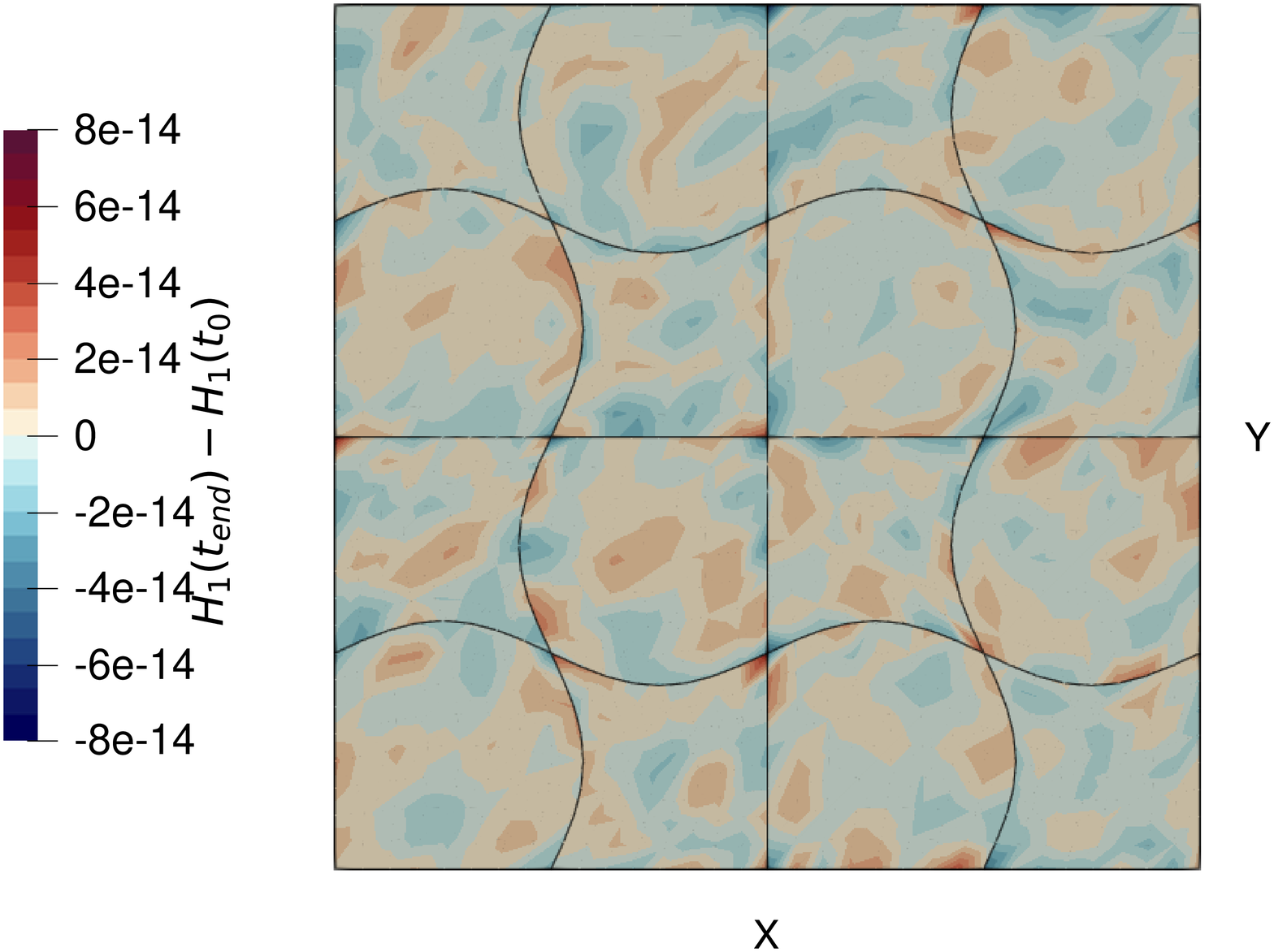}
	}
	\caption{(a) Discontinuous bottom topography in element $3\times2$ and (b) lake-at-rest error at time $t_{end}=100$ for the EC flux with $N=8$ and $\text{CFL}=0.7$}
	\label{fig:well_balanced}
\end{figure}

\subsection{Well-balanced Perturbation}
Next, we introduce a small discontinuous perturbation to the lake-at-rest configuration described in Section \ref{sec:results_well_balanced}, and examine the resulting change in entropy for both the EC and ES schemes. We adopt the previous setup, but modify the initial conditions \eqref{eq:well_balanced_IC} to include a discontinuous perturbation in $h_1$ within a single element
\begin{equation}
	H_1 = \begin{cases} 
		0.65 	&	\text{for element 2$\times$3} \\
		0.6		&	\text{else.}
	\end{cases}
	\label{eq:IC_Perturbation}
\end{equation}

To demonstrate entropy conservation for the approximation \eqref{eq:ec_discretization} and entropy stability for \eqref{eq:es_discretization}, we evaluate the total discrete entropy change within the domain $\Omega$ as
\begin{equation}
	\dot{S}_{\Omega} := \int\limits_{\Omega}\dot{S}\,dV 
	\approx \sum\limits_{i,j=0}^N J_{ij}\omega_{ij}\dot{S}_{ij}.
\end{equation}

Minimum, mean and maximum values of the entropy change for $\smash{\bigstatevec{F}^{EC}}$ and $\smash{\bigstatevec{F}^{ES}}$ within the time interval $T=[0,0.1]$ are presented in Table \ref{tab:entropy_conservation_test}.
For the EC flux we see that entropy is conserved discretely with a total discrete entropy change around machine precision. For the ES flux on the other hand, we observe a noticeable change as entropy is dissipated at discontinuities. The solutions shows a strict decrease in entropy throughout the time interval $T$ and therefore satisfies the entropy inequality \eqref{eq:entropy_inequality}. 
\begin{table}[htb]
	\centering
	\caption{Mean, minimum and maximum values of total discrete entropy change for the EC and ES fluxes in the domain $\Omega$ within $T=[0,1]$. Results obtained for $\text{CFL}=0.7$ and polynomial degree $N=8$.}
	\label{tab:entropy_conservation_test}
	\begin{tabular}{c|l l l}
		\toprule
		&	$\underset{t \in T}{\operatorname{min}} \dot{S}_{\Omega}$			& 	$\underset{t \in T}{\operatorname{mean}} \dot{S}_{\Omega}$			& $\underset{t \in T}{\operatorname{max}} \dot{S}_{\Omega}$ \\ 
		\midrule
		$\bigstatevec{F}^{EC}$	&	$-3.670\cdot10^{-16}$	&	$-1.194\cdot10^{-17}$	&	$\hspace{0.215cm}3.756\cdot10^{-16}$\\
		$\bigstatevec{F}^{ES}$ &	$-2.082\cdot10^{-2}$	&	$-1.273\cdot10^{-4}$	&	$-1.061\cdot10^{-5}$\\
		\bottomrule
	\end{tabular}
\end{table}

\subsection{Parabolic dam break}
As a final test case, we apply the EC and ES schemes to a more complex problem with non-periodic BCs and demonstrate the solution behavior. We consider a partial dam break problem with initial conditions similar to \cite{paz2011local,lee2011fast}, but with a modified geometry featuring a parabolic dam to create a curvilinear boundary. The solution domain is given by $\Omega=[0,10]^2$, and includes a parabolic dam described by the centerline $x=\frac{1}{25}y^2 - 0.4y + 6$, thickness $d=0.2$ and a gap size of $1$. We solve the problem on a curvilinear quadrilateral mesh containing 2239 elements and set the boundary conditions as slip-wall for the dam and Dirichlet for the outer boundaries. The curvilinear quadrilateral mesh was constructed using HOHQMesh.jl\footnote{https://github.com/trixi-framework/HOHQMesh.jl}. Initial conditions are chosen such that both layers are at rest, with a discontinuity across the dam.
\begin{equation}
	\begin{aligned}
			&h_1, h_2 = 
			\begin{cases}
				1.0  & \text{if} \quad  x \leq \frac{1}{25}y^2 - 0.4y + 6\\
				0.75 & \text{else}
			\end{cases}\\
			&\frac{\rho_1}{\rho_2} = 0.25, \quad g = 1, \quad u_1,u_2,v_1,v_2 = 0.
	\end{aligned}
\end{equation}

In Figure \ref{fig:dam_break_over_time} a sequence of computational results for the ES scheme obtained with polynomial degree $N=3$ is shown over time. After the initial dam break we observe a shock and rarefaction wave in both layers and the formation of vortices at the corners of the dam opening. Furthermore, in Figure \ref{fig:comparison_ec_es} we compare the result with the EC scheme without dissipation along the horizontal line $y=5$. We see that while both schemes do well in resolving the shock and rarefaction waves. The dissipation-free EC discretization shows significant spurious oscillations, due to missing entropy dissipation at shocks. The ES scheme on the other hand dissipates at shocks and eliminates notable oscillations. However, the additional dissipation is only designed to satisfy an entropy inequality and additional shock capturing is necessary to obtain a scheme that is oscillation free, while remaining ES, eg. \cite{hennemann2021provably,rueda2021entropy}\\
\begin{figure}[!htb]
	\subfigure[$T=0$]{
		\includegraphics[width=0.49\linewidth,trim={7cm 3cm 12cm 5cm},clip]{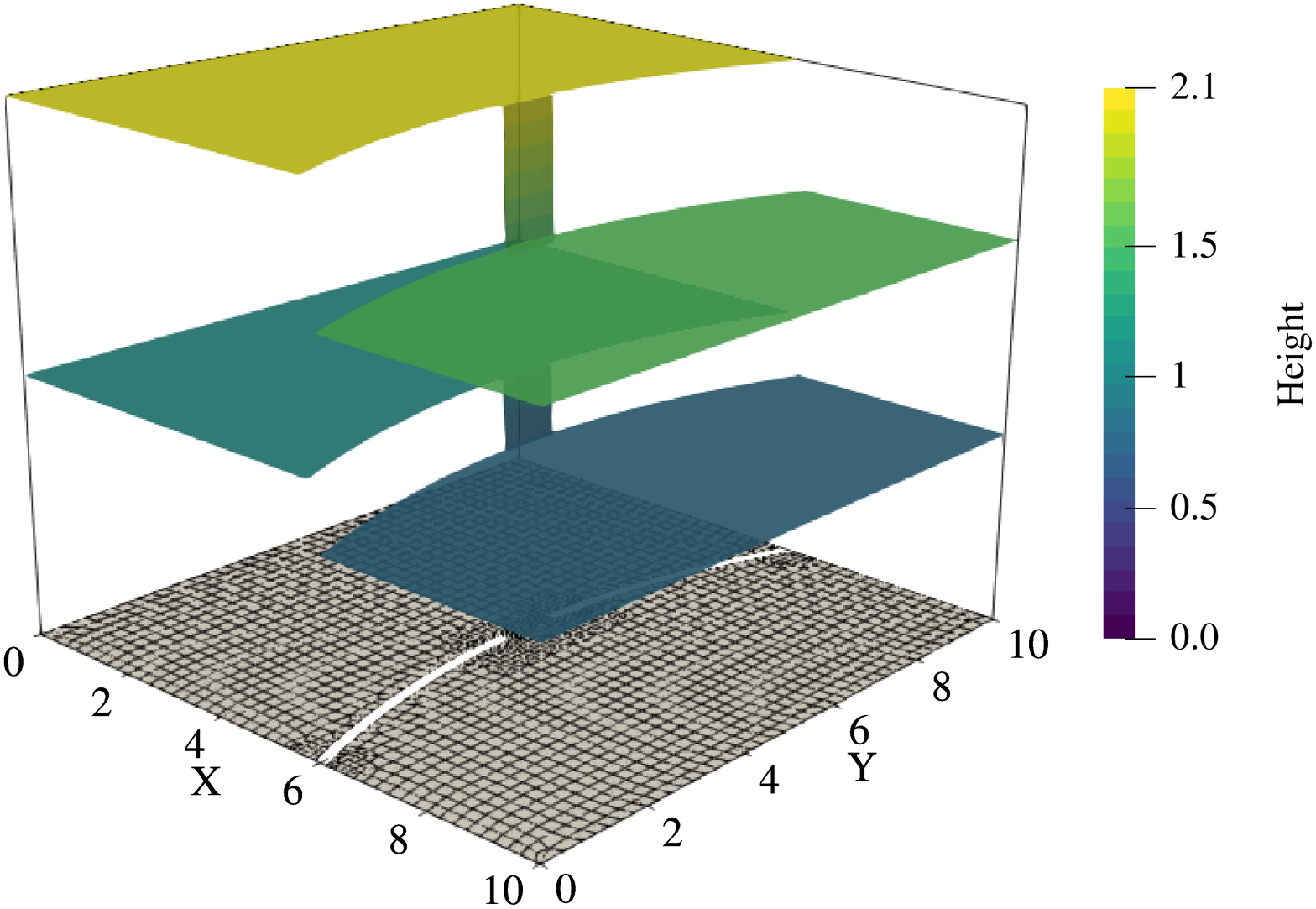}}
	\subfigure[$T=1$]{
		\includegraphics[width=0.49\linewidth,trim={7cm 3cm 12cm 5cm},clip]{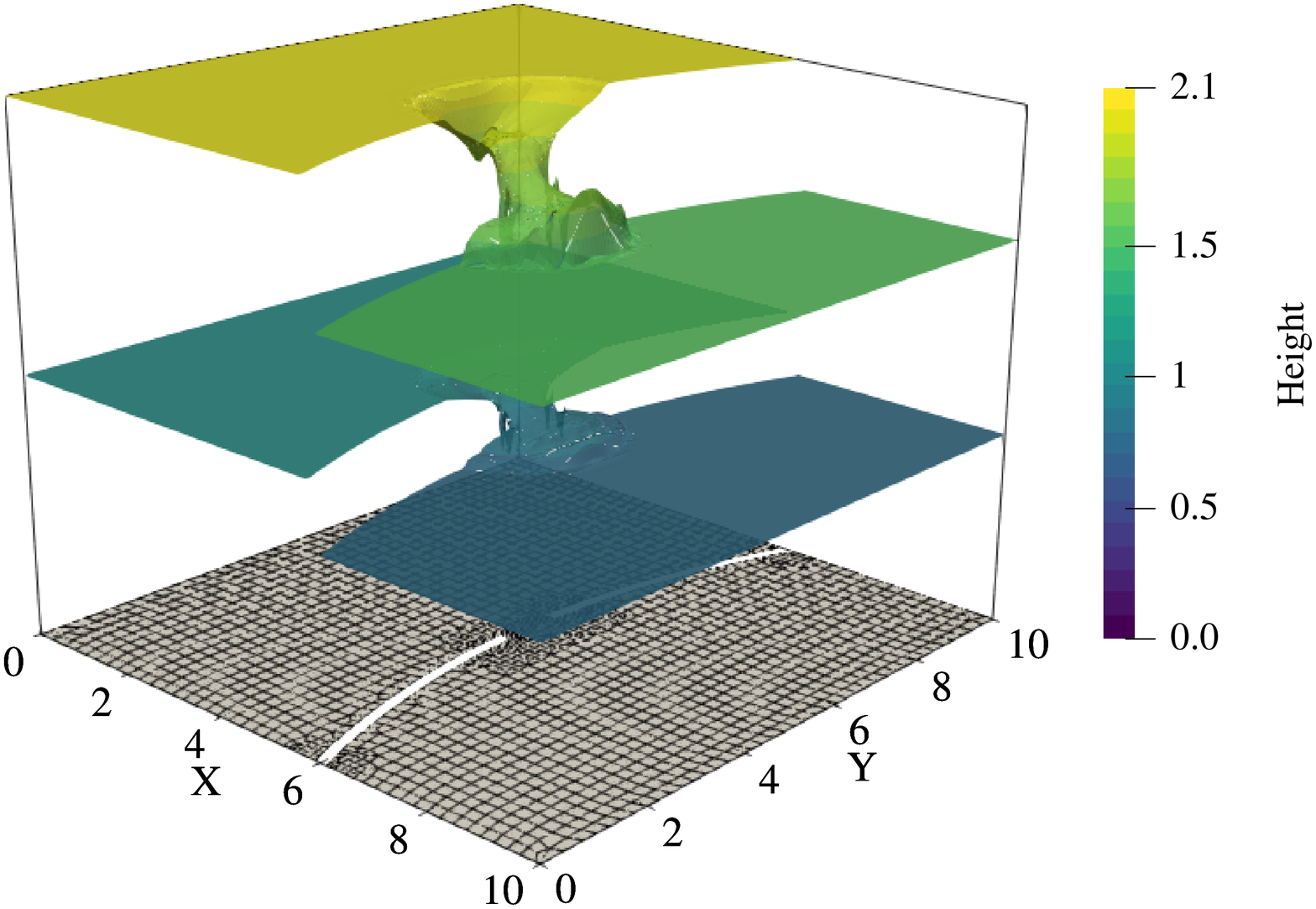}}
	\subfigure[$T=2$]{
		\includegraphics[width=0.49\linewidth,trim={7cm 3cm 12cm 5cm},clip]{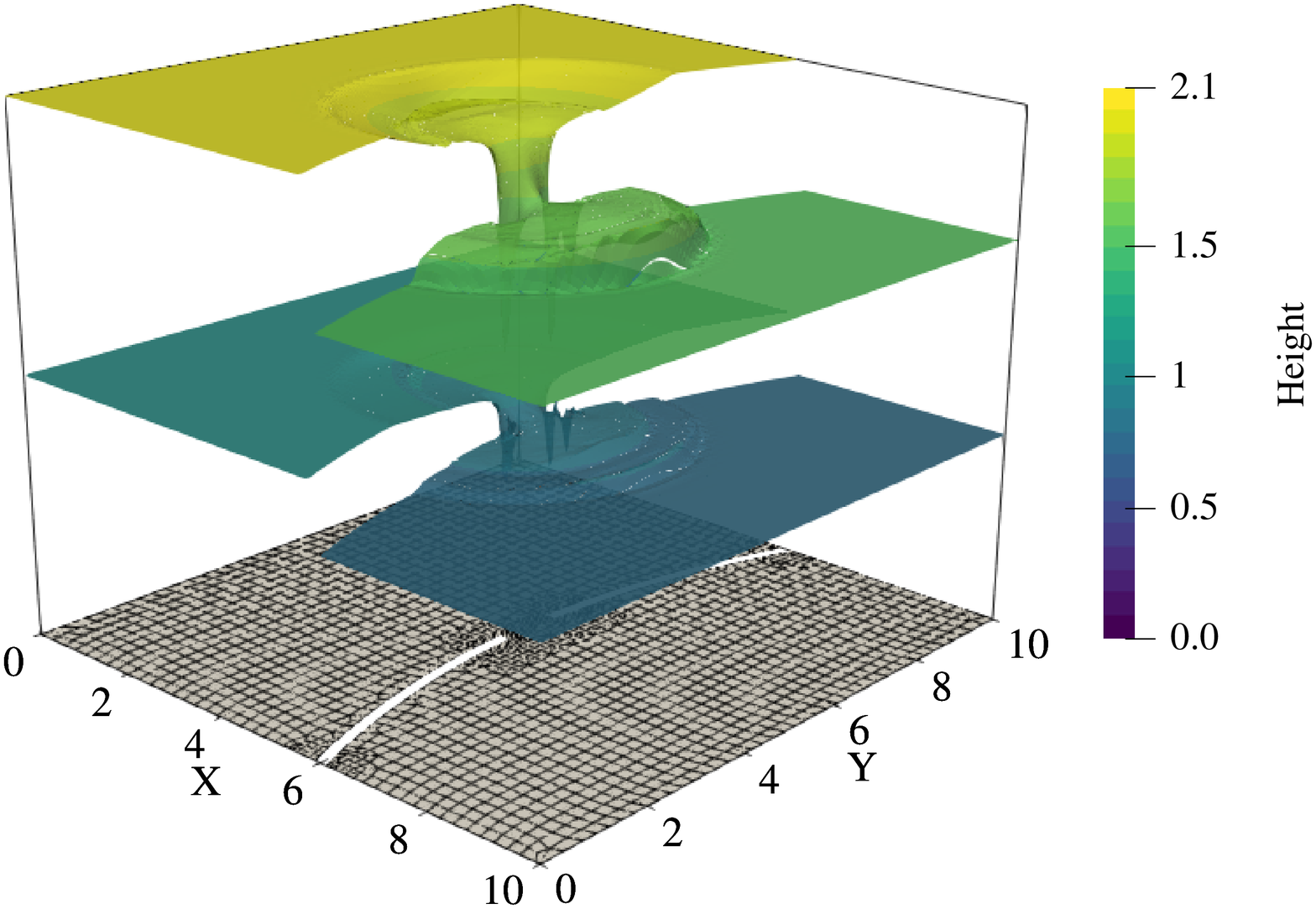}}
	\subfigure[$T=3$]{
		\includegraphics[width=0.49\linewidth,trim={7cm 3cm 12cm 5cm},clip]{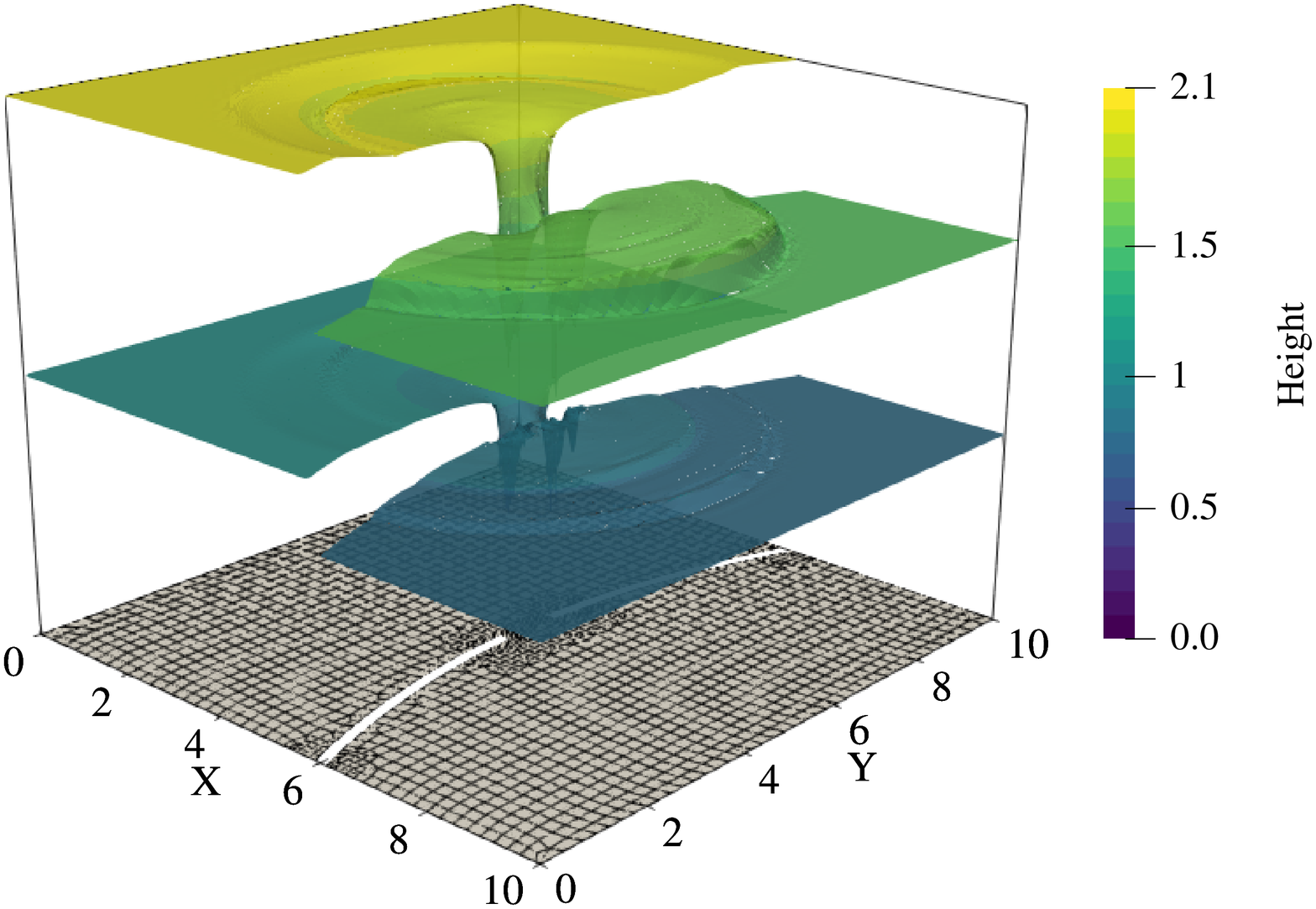}}
	\caption{Visualization of the dam break test case at different times for polynomial degree $N=3$ and $\Delta t=4\cdot10^{-4}$} 
	\label{fig:dam_break_over_time}
\end{figure}
\begin{figure}[!htb]
	\centering
	\includegraphics[width=0.7\linewidth]{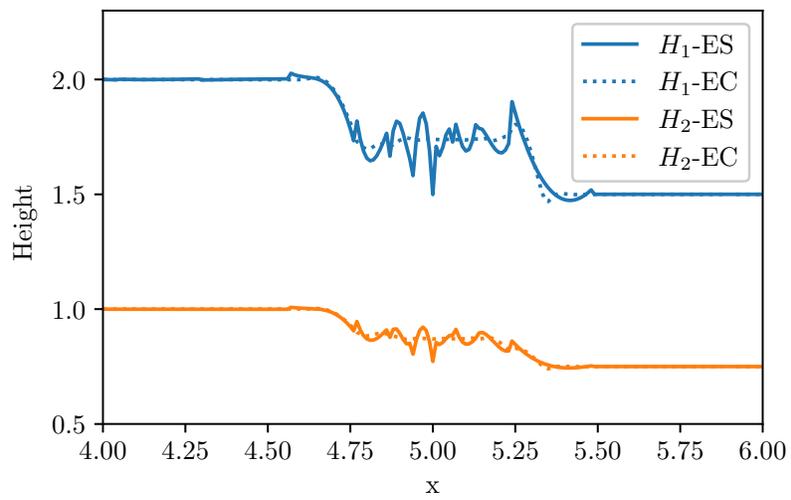}
	\caption{Layer heights along $y=5$ for the EC and ES flux at $t=0.25$ with polynomial degree $N=3$ and $\Delta t=10^{-4}$}
	\label{fig:comparison_ec_es}
\end{figure}

\FloatBarrier

\section{Conclusion}
We presented a high-order collocated nodal discontinuous Galerkin spectral element method (DGSEM) for the two-layer shallow water equations on curvilinear quadrilateral meshes. We first constructed the DGSEM in a specific way such that it is endowed with the summation-by-parts (SBP) property and introduced a path-con\-ser\-va\-tive approximation of the nonconservative terms.
Using the SBP property we then replaced the volume contributions with a flux-differencing formulation and demonstrated that the resulting approximation is entropy conservative (EC) for a specific combination of the EC flux and numerical nonconservative terms.
Applying an equivalent discretization for the pressure and nonconservative terms, we further show that the method is well-balanced for discontinuous bottom topography.
From the dissipation free EC formulation we then constructed an entropy stable scheme by adding additional numerical dissipation at interfaces using a local Lax-Friedrichs type dissipation as a closed form of the eigenvalues is not available.

Finally, we provide numerical results for a number of academic test cases to demonstrate convergence, well-bal\-an\-ced\-ness and entropy stability of the scheme. We then conclude with the application on a more complex parabolic dam break test and demonstrate the solution behavior. The numerical tests verify the analysis and show that the scheme is well-balanced and EC up to machine precision. 

In future work, we aim to complement the scheme with wetting and drying techniques and add shock capturing methods to create a scheme that is oscillation-free. Furthermore, we aim to extend the present formulation to obtain an entropy stable and well-balanced discretization of the Savage-Hutter model \cite{savage1989motion} for submarine avalanches. 

\section*{Acknowledgments}
	Patrick Ersing and Andrew Winters were funded through Vetenskapsrådet, Sweden grant agreement 2020-03642$\,\text{VR}$. Some computations were enabled by resources provided by the National Academic Infrastructure for Supercomputing in Sweden (NAISS) at Tetralith partially funded by the Swedish Research Council through grant agreement no. 2022-06725.

\bibliographystyle{plain}      
\bibliography{twolayer_bibliography}

\begin{appendices}

\section{Entropy Conservation Condition of the EC Flux}\label{sec:entropy_condition_volume}
To show that the EC flux \eqref{eq:volume_flux_ec} satisfies the entropy conservation condition \eqref{eq:entropy_conservation_condition}, we contract EC flux with the jump in entropy variables and restrict the derivation to the $\xi$-direction for clarity as the approach in $\eta$-direction is analogous. The contraction of the EC flux in the $\xi$-direction is given by
\begin{equation}
	\begin{aligned}
		\jump{\statevec{w}}^T \statevec{f_1}^{EC} = 
		&\jump{\rho_1\left(g\left(h_1 + h_2 + b\right) - \frac{u_1^2}{2}-\frac{v_1^2}{2}\right)}\avg{h_1u_1} 
		\\+&\jump{\rho_1v_1}\avg{h_1u_1}\avg{v_1}
		+
		\jump{\rho_1u_1}\left(\avg{h_1u_1}\avg{u_1} + g\avg{h_1}^2 - \frac{g}{2}\avg{h_1^2}\right)
		\\
		+&\jump{\rho_2\left(g\left(\frac{\rho_1}{\rho_2}h_1 + h_2 + b\right) - \frac{u_2^2}{2} - \frac{v_2^2}{2}\right)}\avg{h_2u_2} + \jump{\rho_2v_2}\avg{h_2u_2}\avg{v_2}\\ &\jump{\rho_2u_2}\left(\avg{h_2u_2}\avg{u_2} + g\avg{h_2}^2 - \frac{g}{2}\avg{h_2^2}\right).
	\end{aligned}
	\label{eq:entropy_condition_volume_step_1}
\end{equation}

We expand all the terms in \eqref{eq:entropy_condition_volume_step_1} and see that several terms cancel directly
\begin{equation}
	\begin{aligned}
		&\rho_1g\jump{h_1}\avg{h_1u_1} + \rho_1g\jump{h_2}\avg{h_1u_1} +
		\rho_1g\jump{b}\avg{h_1u_1}
		\\
		-&\frac{1}{2}\rho_1\jump{u_1^2+\cancel{v_1^2}}\avg{h_1u_1} 
		+\rho_1\jump{u_1}\avg{h_1u_1}\avg{u_1} + 
		\rho_1g\jump{u_1}\avg{h_1}\avg{h_1} 
		\\
		-&\rho_1\frac{g}{2}\jump{u_1}\avg{h_1^2}+
		\cancel{\rho_1\jump{v_1}\avg{h_1u_1}\avg{v_1}}+ 
		\rho_1g\jump{h_1}\avg{h_2u_2}
		\\
		+&\rho_2g\jump{h_2}\avg{h_2u_2} +
		\rho_2g\jump{b}\avg{h_2u_2} -
		\frac{1}{2}\rho_2\jump{u_2^2 + \cancel{v_2^2}}\avg{h_2u_2}
		\\
		+ &\rho_2\jump{u_2}\avg{h_2u_2}\avg{u_2} + 
		\rho_2g\jump{u_2}\avg{h_2}\avg{h_2} -
		\rho_2\frac{g}{2}\jump{u_2}\avg{h_2^2}
		\\
		+&\cancel{\rho_2\jump{v_2}\avg{h_2u_2}\avg{v_2}}.
	\end{aligned}
\end{equation}
Then we apply the jump rule $\jump{a^2} = 2\avg{a}\jump{a}$ and again cancel out the resulting equal terms
\begin{equation}
	\begin{aligned}
		&\rho_1g\jump{h_1}\avg{h_1u_1} + \rho_1g\jump{h_2}\avg{h_1u_1} +
		\rho_1g\jump{b}\avg{h_1u_1} 
		\\
		-&\cancel{\rho_1\jump{u_1}\avg{u_1}\avg{h_1u_1}}
		+ \cancel{\rho_1\jump{u_1}\avg{h_1u_1}\avg{u_1}} + \rho_1g\jump{u_1}\avg{h_1}\avg{h_1} 
		\\
		-&\rho_1\frac{g}{2}\jump{u_1}\avg{h_1^2}
		+ \rho_1g\jump{h_1}\avg{h_2u_2} + \rho_2g\jump{h_2}\avg{h_2u_2} 
		\\
		+&\rho_2g\jump{b}\avg{h_2u_2} -
		\cancel{\rho_2\jump{u_2}\avg{u_2}\avg{h_2u_2}}
		+ \cancel{\rho_2\jump{u_2}\avg{h_2u_2}\avg{u_2}}
		\\
		+&\rho_2g\jump{u_2}\avg{h_2}\avg{h_2} -
		\rho_2\frac{g}{2}\jump{u_2}\avg{h_2^2}.
	\end{aligned}
	\label{eq:proof_entropy_condition_eqn_1}
\end{equation}

We continue and group the expression for common variables 
\begin{equation}
	\begin{aligned}
		&\rho_1g\jump{h_1}\avg{h_1u_1} + \rho_1g\jump{u_1}\avg{h_1}\avg{h_1} -
		\rho_1\frac{g}{2}\jump{u_1}\avg{h_1^2}
		\\
		+&\rho_2g\jump{h_2}\avg{h_2u_2} + \rho_2g\jump{u_2}\avg{h_2}\avg{h_2} - \rho_2\frac{g}{2}\jump{u_2}\avg{h_2^2}
		\\
		+&\rho_1g\jump{h_2}\avg{h_1u_1} + \rho_1g\jump{h_1}\avg{h_2u_2}
		\\
		+&\rho_1g\jump{b}\avg{h_1u_1} + \rho_2g\jump{b}\avg{h_2u_2},
	\end{aligned}
\end{equation}
and then make use of an additional rule that holds for jumps and averages of two arbitrary quantities $a$ and $b$ given by
\begin{equation}
	\jump{a}\avg{ab} + \jump{b}\avg{a}^2 - \frac{1}{2}\jump{b}\avg{a^2} = \frac{1}{2}\jump{a^2b}.
	\label{eq:additional_rule_entropy_condition_volume}
\end{equation}

Applying this rule to \eqref{eq:proof_entropy_condition_eqn_1} now shows that the entropy conservation condition \eqref{eq:entropy_conservation_condition} holds
\begin{equation}
	\begin{aligned}
		&\frac{1}{2}\rho_1g\jump{h_1^2u_1} + \frac{1}{2}\rho_1g\jump{h_2^2u_2}
		\\&+\rho_1g\jump{h_2 + b}\avg{h_1u_1}
		+\rho_2g\jump{b + \tfrac{\rho_2}{\rho_1}h_1}\avg{h_2u_2}
		\\ =
		&\jump{\Psi_1} - \avg{\statevec{w}\circ\statevecGreek{\phi}}^T\jump{\statevec{r_1}}
	\end{aligned}
\end{equation}	

\section{Entropy Volume Contribution}\label{sec:proof_volume_contribution_entropy}
To show that the volume contributions generate entropy fluxes at the boundary, we will examine the volume contributions given by
\begin{equation}
	\begin{aligned}
 		\iprodN{\vec{\mathbb{D}}\cdot\bigcontravec{F}^{EC} ,\statevec{W}} + \iprodN{\spacevec{\mathbb{D}}\cdot(\statevecGreek{\Phi}\circ\bigcontravec{R})^{\diamond},\statevec{W}}.
	\end{aligned}
\end{equation}
First we expand the volume contributions of the fluxes
\begin{equation}
	\begin{aligned}
		\iprodN{\vec{\mathbb{D}}\cdot\bigcontravec{F}^{EC} ,\statevec{W}} = 
		\sum\limits_{i,j=0}^N \omega_{ij} \statevec{W}_{ij}^T
		&\left[2\sum_{m=0}^{N} \mathcal{D}_{im} \left(\bigstatevec{F}^{EC}(\statevec{U}_{ij}, \statevec{U}_{mj}) \cdot \avg{J\vec{a}^{\,1}}_{(i,m)j} \right)\right.
		\\&\hspace{-0.18cm}+
		\left.
		2\sum_{m=0}^{N} \mathcal{D}_{jm}   \left(\bigstatevec{F}^{EC}(\statevec{U}_{ij}, \statevec{U}_{im}) \cdot \avg{J\vec{a}^{\,2}}_{i(j,m)} \right)\right]
	\end{aligned}
	\label{eq:volume_flux_term_expanded}
\end{equation}
and the nonconservative terms
\begin{equation}
	\begin{aligned}
		\iprodN{\spacevec{\mathbb{D}}\cdot(\statevecGreek{\Phi}\circ\bigcontravec{R})^{\diamond},\statevec{W}}
		=
		\sum\limits_{i,j=0}^N \omega_{ij} \statevec{W}_{ij}^T
		&\left[2\sum_{m=0}^{N} \mathcal{D}_{im} \left(\left(\frac{1}{2}\statevecGreek{\Phi_{ij}}\circ\jump{\bigstatevec{R}}_{(i,m),j}\right) \cdot \avg{J\vec{a}^{\,1}}_{(i,m)j}\right) \right.
		\\ 
		&\hspace{-0.18cm}+\left.
		2\sum_{m=0}^{N} \mathcal{D}_{jm}   \left(\left(\frac{1}{2}\statevecGreek{\Phi_{ij}}\circ\jump{\bigstatevec{R}}_{i(j,m)}\right)\cdot \avg{J\vec{a}^{\,2}}_{i(j,m)} \right)\right].
	\end{aligned}
	\label{eq:volume_nonconservative_expanded}
\end{equation}
We will only present the proof for the $\xi$-direction as it simplifies the analysis and the $\eta$-direction is done analogously. Next, we rewrite \eqref{eq:volume_flux_term_expanded} in terms of the undivided differencing operator $\mathcal{Q}_{im} = \omega_{i}\mathcal{D}_{im}$
\begin{equation}
	\begin{aligned}
		\sum\limits_{i,j=0}^N \omega_{ij} \statevec{W}_{ij}^T 2&\sum_{m=0}^{N} \mathcal{D}_{im} \left(\bigstatevec{F}^{EC}(\statevec{U}_{ij}, \statevec{U}_{mj}) \cdot \avg{J\vec{a}^{\,1}}_{(i,m)j} \right)
		\\=
		&\sum\limits_{j=0}^N \omega_{j} \sum\limits_{i=0}^N\statevec{W}_{ij}^T \sum_{m=0}^{N} 2\omega_i\mathcal{D}_{im} \left(\bigstatevec{F}^{EC}(\statevec{U}_{ij}, \statevec{U}_{mj}) \cdot \avg{J\vec{a}^{\,1}}_{(i,m)j} \right)
		\\=
		&\sum\limits_{j=0}^N \omega_{j} \sum\limits_{i=0}^N\statevec{W}_{ij}^T \sum_{m=0}^{N} 2 \mathcal{Q}_{im} \left(\bigstatevec{F}^{EC}(\statevec{U}_{ij}, \statevec{U}_{mj}) \cdot \avg{J\vec{a}^{\,1}}_{(i,m)j} \right).
	\end{aligned}
\end{equation}
We make use of the SBP property $2\mathcal{Q}_{im} = \mathcal{Q}_{im} - \mathcal{Q}_{mi} + \mathcal{B}_{im}$ and subsequently exchange the indexing of $i$ and $m$ to rewrite $\mathcal{Q}_{mi}$ as $\mathcal{Q}_{im}$, using that $\bigstatevec{F}^{EC}(\statevec{U}_{ij}, \statevec{U}_{mj})$ and $\avg{J\vec{a}^{\,1}}_{(i,m)j}$ are symmetric regarding $i$ and $j$
\begin{equation}
	\begin{aligned}
		&\sum\limits_{i=0}^N\statevec{W}_{ij}^T \sum_{m=0}^{N} 2 \mathcal{Q}_{im} \left(\bigstatevec{F}^{EC}(\statevec{U}_{ij}, \statevec{U}_{mj}) \cdot \avg{J\vec{a}^{\,1}}_{(i,m)j} \right)
		\\=\hspace{0.14cm}
		&\hspace{-0.14cm}\sum\limits_{i,m=0}^{N} \statevec{W}_{ij}^T\left( \mathcal{Q}_{im} - \mathcal{Q}_{mi} + \mathcal{B}_{im} \right) \left(\bigstatevec{F}^{EC}(\statevec{U}_{ij}, \statevec{U}_{mj}) \cdot \avg{J\vec{a}^{\,1}}_{(i,m)j} \right)
		\\=\hspace{0.14cm}
		&\hspace{-0.14cm}\sum\limits_{i,m=0}^{N} \mathcal{Q}_{im}\left( \statevec{W}_{ij}  - \statevec{W}_{mj} \right)^T \left(\bigstatevec{F}^{EC}(\statevec{U}_{ij}, \statevec{U}_{mj}) \cdot \avg{J\vec{a}^{\,1}}_{(i,m)j} \right)
		\\ &\hspace{0.27cm}+ \mathcal{B}_{im}\statevec{W}_{ij}^T\left(\bigstatevec{F}^{EC}(\statevec{U}_{ij}, \statevec{U}_{mj}) \cdot \avg{J\vec{a}^{\,1}}_{(i,m)j} \right).
	\end{aligned}
	\label{eq:volume_proof_step_1}
\end{equation}
We then use consistency with the physical flux $\bigstatevec{F}^{EC} = \bigstatevec{F}$ at the boundary part, as the boundary matrix consists of zero entries besides $\mathcal{B}_{00}=-1$ and $\mathcal{B}_{NN}=1$ to substitute the definition of the entropy flux potential \eqref{eq:entropy_flux_potential}
\begin{equation}
	\mathcal{B}_{im}\statevec{W}_{ij}^T\bigstatevec{F}^{EC}(\statevec{U}_{ij}, \statevec{U}_{mj}) 
	= 
	\mathcal{B}_{im}\left(
	\spacevec{\Psi}_{ij} + \spacevec{F}^{\,S}_{ij}
	\right)
	\label{eq:entropy_condition_substitution_B}
\end{equation}
and then substitute this in \eqref{eq:volume_proof_step_1} to obtain
\begin{equation}
	\begin{aligned}
		&\sum\limits_{i=0}^N\statevec{W}_{ij}^T \sum_{m=0}^{N} 2 \mathcal{Q}_{im} \left(\bigstatevec{F}^{EC}(\statevec{U}_{ij}, \statevec{U}_{mj}) \cdot \avg{J\vec{a}^{\,1}}_{(i,m)j} \right)
		\\=\hspace{0.14cm}
		&\hspace{-0.14cm}\sum_{i,m=0}^{N} \mathcal{Q}_{im}\left( \statevec{W}_{ij}  - \statevec{W}_{mj} \right)^T \left(\bigstatevec{F}^{EC}(\statevec{U}_{ij}, \statevec{U}_{mj}) \cdot \avg{J\vec{a}^{\,1}}_{(i,m)j} \right)
		\\ &\hspace{0.27cm}+ 
		\mathcal{B}_{im}\left(
		\spacevec{\Psi}_{ij} + {\spacevec{F}}^{\,S}_{ij}
		\right) \cdot \avg{J\vec{a}^{\,1}}_{(i,m)j} 
	\end{aligned}.
	\label{eq:volume_proof_flux_xi_final}
\end{equation}

We then examine the contribution of the nonconservative term volume terms \eqref{eq:volume_nonconservative_expanded} in $\xi$-direction. Again, we use the undivided differencing operator and rewrite the numerical nonconservative term
\begin{equation}
	\begin{aligned}
		&\sum\limits_{i,j=0}^N \omega_{ij} \statevec{W}_{ij}^T
		2\sum_{m=0}^{N} \mathcal{D}_{im} \left(\left(\frac{1}{2}\statevecGreek{\Phi_{ij}}\circ\jump{\bigstatevec{R}}_{(i,m),j}\right) \cdot \avg{J\vec{a}^{\,1}}_{(i,m)j}\right)
		\\=
		&\sum\limits_{j=0}^N \omega_{j} \sum\limits_{i=0}^{N}\statevec{W}_{ij}^T
		\sum_{m=0}^{N} 2\mathcal{Q}_{im} \left(\left(\frac{1}{2}\statevecGreek{\Phi_{ij}}\circ\left(\bigstatevec{R}_{mj} - \bigstatevec{R}_{ij}\right)\right) \cdot \avg{J\vec{a}^{\,1}}_{(i,m)j}\right)
	\end{aligned}
\end{equation}
We apply the SBP property $2\mathcal{Q}_{im} = \mathcal{Q}_{im} - \mathcal{Q}_{mi} + \mathcal{B}_{im}$, exchange indices $i$ and $m$ to rewrite the second operator and regroup the terms
\begin{equation}
	\begin{aligned}
		\sum\limits_{i=0}^{N}\statevec{W}_{ij}^T
		&\sum_{m=0}^{N} 2\mathcal{Q}_{im} \left(\left(\frac{1}{2}\statevecGreek{\Phi_{ij}}\circ\left(\bigstatevec{R}_{mj} - \bigstatevec{R}_{ij}\right) \right) \cdot \avg{J\vec{a}^{\,1}}_{(i,m)j}\right)
		\\=
		\sum\limits_{i=0}^{N}\statevec{W}_{ij}^T
		&\sum_{m=0}^{N} \left(\mathcal{Q}_{im} - \mathcal{Q}_{mi} + \mathcal{B}_{im}\right)
		\left(\left(\frac{1}{2}\statevecGreek{\Phi_{ij}}\circ\left(\bigstatevec{R}_{mj} - \bigstatevec{R}_{ij}\right) \right) \cdot \avg{J\vec{a}^{\,1}}_{(i,m)j}\right)
		\\=
		&\sum_{i,m=0}^{N} \mathcal{Q}_{im} \left( \frac{1}{2}\left(\statevec{W}_{ij}\circ\statevecGreek{\Phi}_{ij} + \statevec{W}_{mj}\circ\statevecGreek{\Phi}_{mj}\right)^T\left(\bigstatevec{R}_{mj} - \bigstatevec{R}_{ij}\right) \cdot\avg{J\vec{a}^{\,1}}_{(i,m)j} \right)
		\\+&\sum_{i,m=0}^{N} \mathcal{B}_{im} \left(\statevec{W}_{ij}^T\left(\frac{1}{2}\statevecGreek{\Phi_{ij}}\circ\left(\bigstatevec{R}_{mj} - \bigstatevec{R}_{ij}\right)\right) \cdot \avg{J\vec{a}^{\,1}}_{(i,m)j}\right).
	\end{aligned}
\end{equation}
We then use the structure of the boundary matrix, which has only zero entries apart from $\mathcal{B}_{00}=-1$ and $\mathcal{B}_{NN}=1$ such that the boundary terms cancel due to the jumps and the nonconservative volume contribution simplifies to 
\begin{equation}
	\begin{aligned}
		\sum\limits_{i=0}^{N}\statevec{W}_{ij}^T
		&\sum_{m=0}^{N} 2\mathcal{Q}_{im} \left(\left(\statevecGreek{\Phi_{ij}}\circ\jump{\bigstatevec{R}}\right) \cdot \avg{J\vec{a}^{\,1}}_{(i,m)j}\right)
		\\=	
		&\sum_{i,m=0}^{N} \mathcal{Q}_{im} \left( \frac{1}{2}\left(\statevec{W}_{ij}\circ\statevecGreek{\Phi}_{ij} + \statevec{W}_{mj}\circ\statevecGreek{\Phi}_{mj}\right)^T\left(\bigstatevec{R}_{mj} - \bigstatevec{R}_{ij}\right) \cdot\avg{J\vec{a}^{\,1}}_{(i,m)j} \right)
	\end{aligned}
\label{eq:volume_proof_noncons_xi_final}
\end{equation}

With the result \eqref{eq:volume_proof_flux_xi_final} we now gather the total volume contribution in $\xi$-direction. From \eqref{eq:volume_proof_flux_xi_final} and \eqref{eq:volume_proof_noncons_xi_final} we substitute the entropy conservation condition \eqref{eq:entropy_conservation_condition} to obtain
\begin{equation}
	\begin{aligned}
		&\sum_{i,m=0}^{N} \mathcal{Q}_{im}\left(\left( \statevec{W}_{ij}  - \statevec{W}_{mj} \right)^T \bigstatevec{F}^{EC}(\statevec{U}_{ij}, \statevec{U}_{mj}) \cdot \avg{J\vec{a}^{\,1}}_{(i,m)j} \right)
		\\
		&\quad+\mathcal{Q}_{im} \left( \frac{1}{2}\left(\statevec{W}_{ij}\circ\statevecGreek{\Phi}_{ij} + \statevec{W}_{mj}\circ\statevecGreek{\Phi}_{mj}\right)^T\left(\bigstatevec{R}_{mj} - \bigstatevec{R}_{ij}\right) \cdot\avg{J\vec{a}^{\,1}}_{(i,m)j} \right)
		\\ &\quad+ 
		\mathcal{B}_{im}\left(
		\spacevec{\Psi}_{ij} + {\spacevec{F}}^{\,S}_{ij}
		\right) \cdot \avg{J\vec{a}^{\,1}}_{(i,m)j}.
		\\=
		&\sum_{i,m=0}^{N} \mathcal{Q}_{im}\left(\left(\spacevec{\Psi}_{ij} - \spacevec{\Psi}_{mj}\right) \cdot \avg{J\vec{a}^{\,1}}_{(i,m)j} \right)
		+ 
		\mathcal{B}_{im}\left(
		\spacevec{\Psi}_{ij} + {\spacevec{F}}^{\,S}_{ij}
		\right) \cdot \avg{J\vec{a}^{\,1}}_{(i,m)j}.
	\end{aligned}
\end{equation}

We then reindex the second term $\spacevec{\Psi}_{mj}$ and apply SBP again to find that the entropy potential vanishes on the boundary part and use that the rows of $\mathcal{Q}$ sum to zero in the last step to obtain
\begin{equation}
	\begin{aligned}
		&\sum_{i,m=0}^{N} \mathcal{Q}_{im}\left(\left(\spacevec{\Psi}_{ij} - \spacevec{\Psi}_{mj}\right) \cdot \avg{J\vec{a}^{\,1}}_{(i,m)j} \right)
		+ 
		\mathcal{B}_{im}\left(
		\spacevec{\Psi}_{ij} + {\spacevec{F}}^{\,S}_{ij}
		\right) \cdot \avg{J\vec{a}^{\,1}}_{(i,m)j}.
		\\=
		&\sum_{i,m=0}^{N} \mathcal{Q}_{im}\left(\left(\spacevec{\Psi}_{ij} + \spacevec{\Psi}_{ij}\right) \cdot \avg{J\vec{a}^{\,1}}_{(i,m)j} \right)
		+ 
		\mathcal{B}_{im}\left(
		\spacevec{\Psi}_{ij} - \spacevec{\Psi}_{ij} + {\spacevec{F}}^{\,S}_{ij}
		\right) \cdot \avg{J\vec{a}^{\,1}}_{(i,m)j}.
		\\=
		&\sum_{i,m=0}^{N} 2\mathcal{Q}_{im}\left(\spacevec{\Psi}_{ij} \cdot \avg{J\vec{a}^{\,1}}_{(i,m)j} \right)
		+ 
		\mathcal{B}_{im}\left( {\spacevec{F}}^{\,S}_{ij} \cdot \avg{J\vec{a}^{\,1}}_{(i,m)j}\right)
		\\=
		&\sum_{i,m=0}^{N} \mathcal{Q}_{im}\left(\spacevec{\Psi}_{ij} \cdot \left(J\vec{a}^{\,1}\right)_{mj} \right)
		+ 
		\mathcal{B}_{im}\left( {\spacevec{F}}^{\,S}_{ij} \cdot \avg{J\vec{a}^{\,1}}_{(i,m)j}\right).
	\end{aligned}
\end{equation}

We use the same approach for the $\eta$-direction and gather the total volume contributions in both directions to have
\begin{equation}
	\begin{aligned}
		&\iprodN{\vec{\mathbb{D}}\cdot\bigcontravec{F}^{EC} ,\statevec{W}} + \iprodN{\spacevec{\mathbb{D}}\cdot(\statevecGreek{\Phi}\circ\bigcontravec{R})^{\diamond},\statevec{W}}
		\\=
		&\sum\limits_{i,j=0}^N \omega_{ij} \sum\limits_{m=0}^{N} \mathcal{D}_{im} \left( \spacevec{\Psi}_{ij} \cdot \left(J\spacevec{a}^1\right)_{mj} \right)
		+\sum\limits_{i,j=0}^N \omega_{ij} \sum\limits_{m=0}^{N} \mathcal{D}_{jm} \left( \spacevec{\Psi}_{ij}\cdot \left(J\spacevec{a}^2\right)_{im} \right)
		\\
		+&\sum\limits_{j=0}^N \omega_{j} \sum\limits_{i,m=0}^{N} \mathcal{B}_{im}\left(
		\spacevec{F}^{\,S}_{ij} \cdot \avg{J\vec{a}^{\,1}}_{(i,m)j}\right)
		+\sum\limits_{j=0}^N \omega_{j} \sum\limits_{i,m=0}^{N} \mathcal{B}_{jm}\left(
		\spacevec{F}^{\,S}_{ij} \cdot \avg{J\vec{a}^{\,1}}_{i(j,m)}\right).
	\end{aligned}
\end{equation}

These are regrouped to separate the contributions from the metric terms
\begin{equation}
	\begin{aligned}
		&\iprodN{\vec{\mathbb{D}}\cdot\bigcontravec{F}^{EC} ,\statevec{W}} + \iprodN{\spacevec{\mathbb{D}}\cdot(\statevecGreek{\Phi}\circ\bigcontravec{R})^{\diamond},\statevec{W}}
		\\=
		&\sum\limits_{i,j=0}^N \omega_{ij} \left(\Psi_1\right)_{ij}
		\left(\sum\limits_{m=0}^{N} \mathcal{D}_{im} \left(Ja_1^1\right)_{mj} + 
		\sum\limits_{m=0}^{N} \mathcal{D}_{jm} \left(Ja_1^2\right)_{im}\right)
		\\
		+&\sum\limits_{i,j=0}^N \omega_{ij} \left(\Psi_2\right)_{ij}
		\left(\sum\limits_{m=0}^{N} \mathcal{D}_{im} \left(Ja_2^1\right)_{mj} + 
		\sum\limits_{m=0}^{N} \mathcal{D}_{jm} \left(Ja_2^2\right)_{im}\right)
		\\
		+&\sum\limits_{j=0}^N \omega_{j} \sum\limits_{i,m=0}^{N} \mathcal{B}_{im}\left(
		\spacevec{F}^{\,S}_{ij} \cdot \avg{J\vec{a}^{\,1}}_{(i,m)j}\right)
		+\sum\limits_{j=0}^N \omega_{j} \sum\limits_{i,m=0}^{N} \mathcal{B}_{jm}\left(
		\spacevec{F}^{\,S}_{ij} \cdot \avg{J\vec{a}^{\,1}}_{i(j,m)}\right).
	\end{aligned}
\end{equation}

This can now be used to show that the volume contributions cancel in entropy space, if the metric identities are satisfied discretely so that

\begin{equation}
	\sum\limits_{l=1}^{2} \frac{\partial}{\partial \xi^l} \mathbb{I}^N \left(Ja_n^l\right)
	=
	\sum\limits_{m=0}^N \mathcal{D}_{im}\left(Ja_n^1\right)_{mj} + \mathcal{D}_{jm}\left(Ja_n^2\right)_{im} = 0.
\end{equation}

From that we see that the volume terms vanish in entropy space and generate the entropy flux at the boundary	
\begin{equation}
	\iprodN{\vec{\mathbb{D}}\cdot\bigcontravec{F}^{EC} ,\statevec{W}} + \iprodN{\spacevec{\mathbb{D}}\cdot(\statevecGreek{\Phi}\circ\bigcontravec{R})^{\diamond},\statevec{W}}
	=
	\int\limits_{\partial E,N} \left(\spacevec{F}^{\,S}\cdot\spacevec{n}\right) \hat{s} \dS.
\end{equation}

\end{appendices}

\end{document}